%% file: article.tex
\documentclass[onefignum,onetabnum]{siamart220329}

\input{shared}

\ifpdf{}
\hypersetup{
  pdftitle={A Scalable Two-Level Domain Decomposition Eigensolver for Periodic Schrödinger Eigenstates in Anisotropically Expanding Domains},
  pdfauthor={Lambert Theisen, Benjamin Stamm}
}
\fi

\begin{document}

\maketitle

\begin{abstract}
  Accelerating iterative eigenvalue algorithms is often achieved by employing a spectral shifting strategy. Unfortunately, improved shifting typically leads to a smaller eigenvalue for the resulting shifted operator, which in turn results in a high condition number of the underlying solution matrix, posing a major challenge for iterative linear solvers. This paper introduces a two-level domain decomposition preconditioner that addresses this issue for the linear Schrödinger eigenvalue problem, even in the presence of a vanishing eigenvalue gap in non-uniform, expanding domains. Since the quasi-optimal shift, which is already available as the solution to a spectral cell problem, is required for the eigenvalue solver, it is logical to also use its associated eigenfunction as a generator to construct a coarse space. We analyze the resulting two-level additive Schwarz preconditioner and obtain a condition number bound that is independent of the domain's anisotropy, despite the need for only one basis function per subdomain for the coarse solver. Several numerical examples are presented to illustrate its flexibility and efficiency.
\end{abstract}



\input{article_content.tex}



\apptocmd{\sloppy}{\vbadness10000\relax}{}{}
\apptocmd{\sloppy}{\hbadness10000\relax}{}{}
\bibliographystyle{siamplain_mod}
\bibliography{refs_betterbibtex_journalabbrevs}

\end{document}

%% file: shared.tex
\usepackage[utf8]{inputenc}

\usepackage{tikz}
\usetikzlibrary{external,patterns}

\usepackage{pgfplots}
\pgfplotsset{compat=1.17}
\usepgfplotslibrary{groupplots,colorbrewer}
\usetikzlibrary{matrix,positioning,arrows,fit,patterns,arrows,arrows.meta}
\usetikzlibrary{fill.image}

\usepackage{tikz-3dplot}

\usepackage{pgfplotstable}

\usepackage[
  backgroundcolor=gray!5,
  linecolor=gray!20,
  bordercolor=gray!20,
  textcolor=gray!70,
  textsize=scriptsize,
]{todonotes}
\usepackage{marginnote}
{}  
\usepackage{ifthen,changepage}
\strictpagecheck{}  
\usepackage{xargs}

\newcommandx{\rtodo}[1]{\tikzexternaldisable\leavevmode\checkoddpage\ifoddpage\normalmarginpar\todo{#1}\else\reversemarginpar\todo{#1}\fi\tikzexternalenable}
\makeatletter
\renewcommand{\todo}[2][]{\tikzexternaldisable\@todo[#1]{#2}\tikzexternalenable} 
\makeatother

\usepackage[caption=false]{subfig}
\usepackage{booktabs}
\usepackage{enumitem}
\usepackage{bbm}
\usepackage[normalem]{ulem}

\usepackage{multirow}

\DeclareMathOperator*{\bigtimes}{\scalerel*{\times}{\textstyle\sum}}
\usepackage{scalerel}

\usepackage{amsfonts}
\usepackage{graphicx}
\usepackage{epstopdf}
\ifpdf{}
  \DeclareGraphicsExtensions{.eps,.pdf,.png,.jpg}
\else
  \DeclareGraphicsExtensions{.eps}
\fi

\usepackage{readarray}

\usepackage{algorithm}
\usepackage{algpseudocode}

\usepackage{silence}
\usepackage{silence}
\newsiamremark{remark}{Remark}
\newsiamremark{example}{Example}
\newsiamremark{hypothesis}{Hypothesis}
\crefname{hypothesis}{Hypothesis}{Hypotheses}
\crefname{lemma}{Lemma}{Lemmata}
\newsiamthm{claim}{Claim}
\newlist{assumptionslist}{enumerate}{1}
\setlist[assumptionslist,1]{wide, labelindent=0pt, font=\bfseries, label={(A\arabic*)}, resume, leftmargin=3em}
\creflabelformat{assumptionslisti}{#2#1#3}
\crefname{assumptionslisti}{assumption}{assumptions}

\usepackage{amsopn}

\DeclareMathOperator*{\argmin}{arg\,min}

  {\end{smallmatrix}\right)} 

\usepackage{bm}

\newcommand{\dd}{\text{d}} 

\headers{Domain Decomposition for Schr\"odinger Equation}{Lambert Theisen and Benjamin Stamm}

\title{A Scalable Two-Level Domain Decomposition Eigensolver for Periodic Schr\"odinger Eigenstates in Anisotropically Expanding Domains\thanks{\LaTeX{} compiled on \today.
\funding{Supported by the German Research Foundation (DFG), project 411724963.}}}

\author{%
Lambert Theisen\thanks{%
  Corresponding author. Applied and Computational Mathematics, RWTH Aachen University, 52062 Aachen, Germany (\email{lambert.theisen@rwth-aachen.de}, \url{https://thsn.dev}) and Institute of Applied Analysis and Numerical Simulation, University of Stuttgart, 70569 Stuttgart, Germany.
}%
\and%
Benjamin Stamm\thanks{%
  Institute of Applied Analysis and Numerical Simulation, University of Stuttgart, 70569 Stuttgart, Germany (\email{best@ians.uni-stuttgart.de}).%
}%
}%

%% file: article_content.tex

\section{Introduction}\label{sec:introduction}
This paper presents a robust and efficient method to solve parametrized linear Schrödinger eigenvalue problems (EVPs) on open bounded domains \(\Omega_L \subset \mathbb{R}^d\) of the form: Find \((\phi_L, \lambda_L)\in (H^1_0(\Omega_L) \setminus \{0\})\times \mathbb R\), such that
\begin{align}\label{eq:schroedingerEquation}
  - \Delta \phi_L + V \phi_L &= \lambda_L \phi_L
  \quad
  \textnormal{ in } \Omega_L
  ,
\end{align}
where \(H^1_0(\Omega_L)\) denotes the standard Sobolev space of index \(1\) with zero Dirichlet trace on \(\partial \Omega_L\). The main focus will be on anisotropically expanding domains \(\Omega_L\), which are modeled by \(d\)-dimensional boxes given by
\begin{align}\label{eq:geometry-description}
  \boldsymbol{z} &\in \Omega_L = {(0, L)}^{p} \times {(0, \ell)}^{q}
  =:
  \Omega_{\boldsymbol{x}} \times \Omega_{\boldsymbol{y}} \subset \mathbb{R}^d
  \quad
  \textnormal{ with }
  L,\ell \in \mathbb{R}
  ,
\end{align}
where the spatial variables are collected as \(\boldsymbol{z} := (\boldsymbol{x}, \boldsymbol{y}) = (x_1,\dots,x_p,y_1,\dots,y_q)\), highlighting the fact that some directions expand with \(L \to \infty\) while the other directions are fixed with \(\ell = \operatorname{const}\). Note that the box setup is only chosen for simplicity of the analysis, and we provide an elegant extension in \cref{sec:numericalExperiments} to the general case. Although \(L \in \mathbb{R}\) in general, we will often assume \(L \in \mathbb{N}\) for simplicity in the analysis and notation. For the external potential \(V\) in \cref{eq:schroedingerEquation}, we assume the following:  
\begin{assumptionslist}
  \item\label{ass:pot-dir-periodic} The potential \(V\) is directional-periodic with a period of \(1\) in each expanding direction: \(V(\boldsymbol{x}, \boldsymbol{y}) = V\left(\boldsymbol{x} + \boldsymbol{i} , \boldsymbol{y}\right)\) for all \((\boldsymbol{x},\boldsymbol{y}) \in \Omega_L, \boldsymbol{i} \in \mathbb{Z}^p\);
  \item\label{ass:pot-ess-bounded} The potential \(V\) is essentially bounded: \(V \in L^\infty(\Omega_L)\).
\end{assumptionslist}
\par
The period of \(1\) in \labelcref{ass:pot-dir-periodic} is only chosen for simplicity and \labelcref{ass:pot-ess-bounded} allows us to assume \(V \ge 0\) a.e.~in \(\Omega_L\) since a constant spectral shift to \(V\) does not affect the eigenfunctions of \cref{eq:schroedingerEquation}.
\cref{sfig:geometricSetup} presents the geometrical framework and indicates the properties of \(V\). After discretization, the problem~\eqref{eq:schroedingerEquation} leads to a large, sparse, generalized algebraic EVP of the form
\begin{equation}\label{eq:algebraicEVP}
  \boldsymbol{A} \boldsymbol{x} = \lambda \boldsymbol{B} \boldsymbol{x}
  ,
\end{equation}
that is then solved with classical iterative algorithms like, for example, the inverse power method (IPM)~\cite{tapiaInverseShiftedInverse2018} or the Locally Optimal Preconditioned Conjugate Gradient (LOPCG)~\cite{knyazevPreconditionedConjugateGradient1991} method. However, the convergence speed of these algorithms depends on the \emph{fundamental ratio} \(r_L := \lambda_L^{(1)} / \lambda_L^{(2)}\). And for \(L \to \infty\), the \emph{fundamental gap} collapses (i.e., \(\lambda_L^{(1)} - \lambda_L^{(2)} \to 0\)), leading to \(r_L \to 1\) as we extensively illustrate in the motivational \cref{sec:motivation}. As a result, the convergence speed of iterative eigensolvers, equal to the fundamental ratio, deteriorates drastically and becomes arbitrarily bad.
\par
This problem was overcome in~\cite{stammQuasiOptimalFactorizationPreconditioner2022} by the quasi-optimal shift-and-invert (QOSI) preconditioner that replaces the original matrix from \cref{eq:algebraicEVP} by a shifted \(\boldsymbol{A}_{\sigma} := \boldsymbol{A} - \sigma \boldsymbol{B}\). The shifting parameter \(\sigma \in \mathbb{R}\) is set to be the \(L\)-asymptotic limit of the first eigenvalue, i.e.~\(\sigma := \lim_{L \to \infty} \lambda_L^{(1)}\), which can be easily computed by solving the same operator on a unit cell domain \(\Omega_1\) where periodic conditions replace all Dirichlet boundary conditions in the expanding \(\boldsymbol{x}\)-direction as it was shown in~\cite{stammQuasiOptimalFactorizationPreconditioner2022}. More precisely, \((\psi, \sigma)\) is the first eigenpair to the problem
\begin{equation}\label{eq:cell-problem}
  \left\{
    \begin{aligned}
      -\Delta \psi + V \psi &= \sigma \psi
      \textnormal{ in } \Omega_1 = {(0,1)}^p \times {(0,\ell)}^q
      ,
      \\
      \psi&=0 \textnormal{ on } \Omega_{\boldsymbol{x}} \times {\{0,\ell\}}^q
      ,
      \\
      \psi &\textnormal{ and } \partial_{x_i}\psi \textnormal{ are } x_i\textnormal{-periodic}
      .
    \end{aligned}
  \right.
\end{equation}
The QOSI preconditioner \(\boldsymbol{A}^{-1}_{\sigma}\) then leads to a uniformly bounded shifted fundamental ratio as, for some \(L^* \in \mathbb{R}\), \(r_{L}(\sigma) := (\lambda_L^{(1)} - \sigma) / (\lambda_L^{(2)} - \sigma) \le C < 1\) for all \(L > L^*\) such that the eigensolver converges in \(k_{\textnormal{it}} \in \mathcal{O}(1)\) iterations (i.e., up to a multiplicative constant optimal).
\par
However, in each eigensolver iteration, c.f.~inverse iteration \(\boldsymbol{x}_{k+1} = \boldsymbol{A}_{\sigma}^{-1} \boldsymbol{B} \boldsymbol{x}_k\), the application of \(\boldsymbol{A}^{-1}_{\sigma}\) amounts to solving shifted linear systems of the form
\begin{equation}\label{eq:abstract-shifted-linear-system}
  (\boldsymbol{A} - \sigma \boldsymbol{B}) \boldsymbol{x} = \boldsymbol{b}
  .
\end{equation}
The catch here is that in the critical limit, \(L \to \infty\), with fixed discretization \(h = \textnormal{const}\), \(\lambda^{(1)}_L \to \sigma\) by construction and thus the condition number \(\kappa(\boldsymbol{A} - \sigma \boldsymbol{B}) = (\lambda_L^{(n)}-\sigma)/(\lambda_L^{(1)}-\sigma) \to \infty\) as observed in~\cite{stammQuasiOptimalFactorizationPreconditioner2022}. The issue of an exploding condition number is dramatic when we employ iterative linear solvers like the conjugate gradient (CG)~\cite{hestenesMethodsConjugateGradients1952} method to solve \cref{eq:abstract-shifted-linear-system} since their convergence rate has now become arbitrarily bad.

\subsection{Our contribution}
This paper fixes these problems by introducing and analyzing a new preconditioner \(\boldsymbol{M}_{\sigma}^{-1}\), now for the shifted linear systems of \cref{eq:abstract-shifted-linear-system}. As it is very natural for the present geometry \(\Omega_L\) and potentials \(V\), the preconditioner uses classical, overlapping Schwarz domain decomposition (DD)~\cite{doleanIntroductionDomainDecomposition2015}. However, a coarse space must be prescribed for robustness and numerical scalability. Surprisingly, it is even directly available since it is based on the spectral asymptotics of the problem. Since the coarse space components are inherently related to the function \(\psi\) of \cref{eq:cell-problem} used in the so-called \emph{factorization principle} (c.f.~\cite{allaireHomogenizationSpectralProblem2000,stammQuasiOptimalFactorizationPreconditioner2022}) to derive \(\sigma\), we call \(\boldsymbol{M}_{\sigma}^{-1}\) the periodic factorization (PerFact) preconditioner. The PerFact-preconditioned shifted linear systems can then also be solved in quasi-optimal iterations since \(\kappa(\boldsymbol{M}_{\sigma}^{-1} \boldsymbol{A}_{\sigma}) \le C\) for all \(L\). As a result, only the unique combination of the QOSI eigen-preconditioner and the PerFact linear-preconditioner, in the end, yields an efficient iterative algorithm that is robust for all domain sizes \(L\).
\par
For the provided analysis, we embed the preconditioner partly in the theory of spectral coarse spaces while still focussing on the unusual case of anisotropically expanding domains. This setup, and the challenging fact that we asymptotically shift with the lowest eigenvalue, required a new SPSD splitting and the introduction of an auxiliary periodic neighborhood decomposition.

\subsection{Motivation: The \emph{shifting dilemma} in the Laplace EVP}\label{sec:motivation}
When it comes down to highlighting the critical difficulty and motivating the present work, the Laplace EVP with \(V(\boldsymbol{z})=0\) in \cref{eq:schroedingerEquation} is the perfect academic example fulfilling \cref{ass:pot-dir-periodic,ass:pot-ess-bounded}. 
For a standard second-order five-point finite difference stencil of a two-dimensional \(\Omega_L\) with \(p,q,\ell=1\), the eigenvalues of the resulting discretization matrix \(\boldsymbol{A}\) are all known. Let \(n\) denote the number of grid points per unit length and introduce the uniform mesh size \(h:=1/(n+1)\). Then all eigenvalues are given (c.f.~the formula in~\cite[p40;p63]{ciaramellaIterativeMethodsPreconditioners2022} using \(2\sin^2(x/2)=1-\cos{x}\)) by
\begin{equation}
  \lambda_{i,j} \in \left\{ \tfrac{4}{h^2} \left( \sin^2(\tfrac{\pi h i}{2L}) + \sin^2(\tfrac{\pi h j}{2}) \right) \middle|\ i \in \{1,\dots,L(n+1)-1\}, j \in \{1,\dots,n\} \right\}
  .
\end{equation}
If we now employ the \(\sigma\)-shifted inverse power method (\(\boldsymbol{x}_{k+1} := \mathcal{R}(\boldsymbol{A}_{\sigma}^{-1} \boldsymbol{x}_k)\) with some retraction \(\mathcal{R}:\mathbb{R}^n \mapsto \mathbb{S}^{n-1}\) onto the unit sphere) for \(\boldsymbol{A}_{\sigma}^{-1} = (\boldsymbol{A} - \sigma \boldsymbol{I})^{-1}\) as the outer iteration, the convergence rate is given by the fundamental ratio~\cite[p366]{golubMatrixComputations2013} of the first two eigenvalues \(\rho_{\textnormal{IPM}} = r_L(\sigma) = \tfrac{\lambda_1 - \sigma}{\lambda_2 - \sigma}\) where (for \(L > \ell = 1\))  
\begin{equation}\label{eq:discrete-laplace-evals12}
  \lambda_1 := \lambda_{1,1} = \tfrac{4}{h^2} \left( \sin^2(\tfrac{\pi h}{2L}) + \sin^2(\tfrac{\pi h}{2}) \right)
  , \;
  \lambda_2 := \lambda_{2,1} = \tfrac{4}{h^2} \left( \sin^2(\tfrac{2 \pi h}{2L}) + \sin^2(\tfrac{\pi h}{2}) \right)
  .
\end{equation}
On the other hand, in each outer iteration, a large, sparse, linear system with \(\boldsymbol{A}_{\sigma}\) needs to be solved. Using the CG method, the convergence rate of that inner solver is given by \(\rho_{\textnormal{CG}} = (\sqrt{\kappa(\boldsymbol{A}_{\sigma})}-1) /(\sqrt{\kappa(\boldsymbol{A}_{\sigma})}+1)\) with
\begin{equation}\label{eq:discrete-laplace-evals-condition-max}
  \kappa(\boldsymbol{A}_{\sigma}) = (\lambda_{\max} - \sigma) / (\lambda_1 - \sigma)
  ,
  \textnormal{ and }
  \lambda_{\max} = \lambda_{L(n+1)-1,n} = \tfrac{4}{h^2} \left( \cos^2(\tfrac{\pi h}{2L}) + \cos^2(\tfrac{\pi h}{2}) \right)
  .
\end{equation}
Now, the error for the first eigenvector in the IPM reduces~\cite[Thm.~8.2.1]{golubMatrixComputations2013} as \(|\sin{\theta_k}| \le \rho_{\textnormal{IPM}}^k \tan{\theta_0}\) with \(\theta_k \in [0,\tfrac{\pi}{2}]\) being the angle of the current iterate \(\boldsymbol{x}_k\) to the real eigenvector \(\boldsymbol{x}^{(1)}\) defined by \(\cos(\theta_k) = |\langle \boldsymbol{x}^{(1)}, \boldsymbol{x}_k \rangle|\). So, in order to decrease the ratio \(|\sin{\theta_k}|/\tan{\theta_0}\) by a factor of \(1/R \in \mathbb{R}\), we need \(n_{\textnormal{IPM}} = - \ln R / \ln \rho_{\textnormal{IPM}}\) iterations where \(\ln(\cdot)\) denotes the natural logarithm. Each application of \(\boldsymbol{A}_{\sigma}\) is then executed using the CG algorithm until the residual reduction \({\|\boldsymbol{r}_k\|}_2 / {\|\boldsymbol{r}_0\|}_2 \le 1/Q \in \mathbb{R}\) is archived, where \(\boldsymbol{r}_k\) is the residual in step \(k\). Using the residual-to-error-energy bound~\cite[p4]{kelleyIterativeMethodsLinear1995} and~\cite[Lem.~C.10]{toselliDomainDecompositionMethods2005}, we obtain the relation
\begin{equation}
  \frac{{\|\boldsymbol{r}_k\|}_2}{{\|\boldsymbol{r}_0\|}_2}
  \le
  \sqrt{\kappa(\boldsymbol{A}_{\sigma})} \frac{{\|\boldsymbol{e}_k\|}_{\boldsymbol{A}_{\sigma}}}{{\|\boldsymbol{e}_0\|}_{\boldsymbol{A}_{\sigma}}}
  \le
  2 \sqrt{\kappa(\boldsymbol{A}_{\sigma})} \rho_{\textnormal{CG}}^k
  ,
\end{equation}
in which \(\boldsymbol{e}_k\) is the error between the current inner iterate and the linear system's solution. Thus, each inner CG algorithm needs \(n_{\textnormal{CG}} = - \ln(2\sqrt{\kappa(\boldsymbol{A}_{\sigma})}Q) / \ln{\rho_{\textnormal{CG}}}\) iterations. Assuming no prior information, equally good initial guesses, and sharpness of the bounds, the inner-outer algorithm, then, in total, would need \(n_{\textnormal{tot}} = n_{\textnormal{IPM}} n_{\textnormal{CG}}\) CG iterations, which is a good measure of the computational cost.
\par
For an algorithm to be efficient, it must be scalable w.r.t.~\(L\). Thus, doubling the system size (i.e., \(L \mapsto 2L\)) while doubling computational resources (i.e., ranks or processes) leads ideally to the same wall clock time needed for the algorithm. Assuming perfect parallelizability of all operations and because the system matrix is twice as large in that case, the wall clock time can only be kept constant if the total number of operations does not increase when doubling \(L\). However, even in this simple example, we observe the \emph{shifting dilemma} for increasing domain lengths \(L \to \infty\) and constant \(h\) since we obtain the following for the different possible shifts:
\par
\textbf{No shift} (\(\sigma=0\)): The application of no shift results in the fundamental ratio converging to 1 since \(\lim_{L \to \infty} \lambda_1/\lambda_2 = 1\) using the values from \cref{eq:discrete-laplace-evals12}. The condition number \(\kappa(\boldsymbol{A})\), on the other hand, converges to a finite value as \(\lim_{L \to \infty}\lambda_{\max} / \lambda_1 = (1 + \cos^2(\tfrac{\pi h}{2})) / \sin^2(\tfrac{\pi h}{2})\) using \cref{eq:discrete-laplace-evals12,eq:discrete-laplace-evals-condition-max}. Thus, for \(L \to \infty\), we obtain \(n_{\textnormal{IPM}} \to \infty\) and \(n_{\textnormal{CG}} \to C > 0\), and the algorithm can not be scalable as \(n_{\textnormal{tot}} \to \infty\).
\par
\textbf{QOSI shift}~\cite{stammQuasiOptimalFactorizationPreconditioner2022} (\(\sigma = \lambda_{\infty} = \tfrac{4}{h^2}\sin^2(\tfrac{\pi h}{2})\)): Using the quasi-optimal shift leads to a uniform bounded ratio \(\tfrac{\lambda_1 - \sigma}{\lambda_2 - \sigma} = \sin^2(\tfrac{\pi h}{2L}) / \sin^2(\tfrac{2 \pi h}{2L}) \to 1/4\) for \(L \to \infty\) using l'H\^{o}pital's rule and thus bounded outer itations with \(n_{\textnormal{IPM}} \le C\). However, the condition number \(\kappa(\boldsymbol{A}_{\sigma})\) of the shifted matrix now explodes when \(L \to \infty\) as \(\lim_{L \to \infty} (\lambda_{\max} - \sigma) / (\lambda_1 - \sigma) > \lim_{L \to \infty} (\cos^2(\tfrac{\pi h}{2L}) / \sin^2(\tfrac{\pi h}{2L})) = \infty\) by the minorant argument. This behavior is natural since we shifted by the asymptotic limit of the first eigenvalue \(\lambda_1\). Thus, for \(L \to \infty\), \(n_{\textnormal{CG}} \to \infty\), and \(n_{\textnormal{tot}} \to \infty\).
\par
\textbf{Adaptive hybrid shift} (\(\sigma \mapsto \sigma_L \in (0,\lambda_\infty)\)): If both extreme cases do not provide a scalable method, one might ask if there is a chance that a (possible \(L\)-dependent) shift in the whole interval \((0,\lambda_{\infty})\) results in \(n_{\textnormal{tot}}\) to be bounded w.r.t.~\(L\). In \cref{sfig:iterations-numbers-laplace-evp}, we evaluated the formulas for \(n_{\textnormal{IPM}}\) and \(n_{\textnormal{CG}}\) using an exemplary setup of \(h=1/10, R=e^7, Q=e^4\) to obtain values for \(n_{\textnormal{tot}}\) depending on the shift \(\sigma\) for different \(L \in \{4,5,6,7\}\). We observe that for all \(\sigma\), the total CG iterations \(n_{\textnormal{tot}}\) grow as \(L \to \infty\). Thus, there is also no sweet spot in between both extremal cases, which suggest that also no adaptive shift can produce a scaling algorithm since all possible shifts still result in \(n_{\textnormal{tot}} \to \infty\).
\par
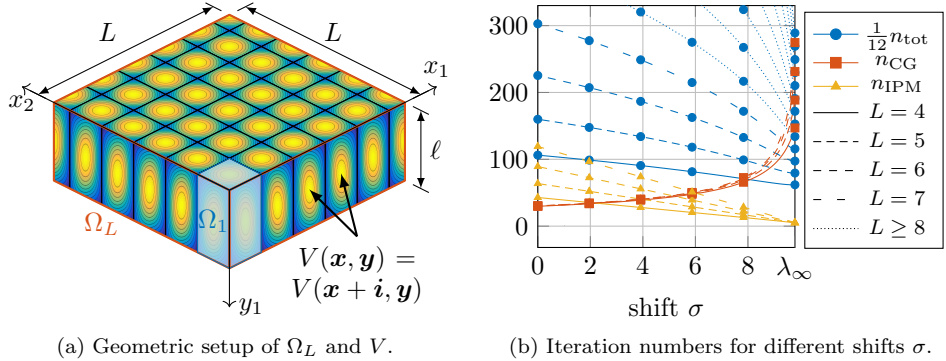
\begin{figure}[t]%
  \centering%
  \subfloat[%
    Geometric setup of \(\Omega_L\) and \(V\).\label{sfig:geometricSetup}%
  ]{%
    \newcommand{\datapath}{./plots/sketches}%
    \input{plots/sketches/geometrical-setup.tex}%
  }%
  \hfill%
  \subfloat[%
    Iteration numbers for different shifts \(\sigma\).\label{sfig:iterations-numbers-laplace-evp}%
  ]{%
    \newcommand{\datapath}{./plots/iteration-numbers-laplace-inner-outer}%
    \input{plots/iteration-numbers-laplace-inner-outer/iteration-numbers-laplace-inner-outer.tex}%
  }
  \caption{%
    \textnormal{\textbf{(a)}} Geometric setup of \(\Omega_L\) with \(p=2\) expanding and \(q=1\) fixed directions with dimensions \(L=5.5,\ell=2\). \textnormal{\textbf{(b)}} Iteration number estimates for an inner-outer eigenvalue algorithm using IPM/CG for the Laplacian EVP on \((0,L)\times(0,1)\) using finite differences (\(h=1/10\)) and different shifts \(\sigma\). Note that the arbitrary scaling of \(n_{\textnormal{tot}}\) is only applied for better visualization.
  }%
\end{figure}%
Thus, the shifting dilemma prevents us from using the QOSI strategy with standard inner solvers without special treatment, and we need to provide something better \textendash\ another preconditioner for the inner CG solver. In our discussion, we intentionally omitted the consideration of the case where \(h \to 0\). Including this scenario would obscure the primary insight that shifting is why we require an alternative preconditioner in this context, in contrast to the classical case of a fixed domain and finer meshes.

\subsection{State-of-the-art and context}\label{ssec:state-of-the-art}
First, the Schrödinger equation \cref{eq:schroedingerEquation} is fundamental in quantum mechanics for describing wave functions \(\phi\) in systems influenced by a potential \(V\). Although suitable for simple systems, its importance extends to complex simulations in computational chemistry. Usual methods like self-consistent field (SCF) iterations~\cite{roothaanNewDevelopmentsMolecular1951, cancesDensityFunctionalTheory2023} rely on solving linear eigenvalue problems, even in nonlinear contexts. One applicable example is the Gross-Pitaevski equation (GPE) for Bose--Einstein condensates, focusing on minimal energy ground states. In contrast to the SCF, direct minimization methods aim to directly obtain ground-state solutions by minimizing the energy. Such methods are, e.g., based on gradient flow~\cite{heidGradientFlowFinite2021, dorichErrorBoundsDiscrete2023, henningSobolevGradientFlow2020}, manifold optimization~\cite{altmannEnergyadaptiveRiemannianOptimization2022, altmannRiemannianNewtonMethods2023, cancesConvergenceAnalysisDirect2021}, preconditioned CG~\cite{gaidamourBEC2HPCHPCSpectral2021}, or other variations of linearized inverse iterations~\cite{altmannJmethodGrossPitaevskiiEigenvalue2021,henningDependencySpectralGaps2023}. Common to these methods is the need to repeatedly solve systems similar to \cref{eq:abstract-shifted-linear-system} \textendash\ either for the update step or to apply preconditioners.
\par
Second, while initially seeming artificial, our anisotropic and periodic geometry is essential for observing exceptional phenomena. One example is the carbon allotope hierarchy. This hierarchy progresses from \emph{0d} fullerene through \emph{1d} carbon nanowires~\cite{zhaoCarbonNanowireMade2003} and nanotubes~\cite{barraultDomainDecompositionElectronic2008}, leading to Nobel-winning research on \emph{2d} graphene~\cite{gerstnerNobelPrize20102010}. Recent advances include \emph{2.5d} materials with Moiré superlattices~\cite{cancesSimpleDerivationMoirescale2023} in twisted bilayer graphene~\cite{caoUnconventionalSuperconductivityMagicangle2018} and the exploration of dimensions above \(d>3\) in the framework of time crystals~\cite{wilczekQuantumTimeCrystals2012} (\(p=3,q=1\) in \cref{eq:geometry-description}). Studying these crystalline structures~\cite{cancesModifiedOperatorMethodCalculation2022,herbstDFTKJulianApproach2021} and searching for promising ones via high-throughput simulations~\cite{mounetTwodimensionalMaterialsHighthroughput2018,herbstRobustEfficientLine2022} is, therefore, increasingly important.
\par
Third, the considered geometries of \cref{eq:geometry-description} naturally harmonize very well with the DD method. In fact, for elliptic source problems, the classical Schwarz method is weakly scalable and does not require coarse correction~\cite{cancesDomainDecompositionImplicit2013} on anisotropic domains with Dirichlet boundaries. This remarkable fact was studied in~\cite{ciaramellaAnalysisParallelSchwarz2017,ciaramellaAnalysisParallelSchwarz2018,ciaramellaAnalysisParallelSchwarz2018a} and further extended in~\cite{ciaramellaScalabilityParallelSchwarz2020,ciaramellaScalabilitySchwarzMethod2020,hassanMathematicalAnalysisBoundary2020,reuskenAnalysisSchwarzDomain2021}, directly inspiring this work with the simple question: What happens to the DD algorithms when we replace the linear problem with an eigenvalue problem for the same operator? \emph{It changes a lot} \textendash\ is the short answer, which we elaborate on in this article.
Using DD in eigenvalue computations has been previously explored through various approaches like MDD~\cite{bencteuxDomainDecompositionElectronic2008a,bencteuxAnalysisQuadraticProgramming2010}, AMLS~\cite{bennighofAutomatedMultilevelSubstructuring2004}, DAC~\cite{chenAnalysisDivideandconquerMethod2016}, and others~\cite{luiRecentResultsDomain1996,luiDomainDecompositionMethods2000,mathewDomainDecompositionMethods2008,xuNovelDomainDecomposition2020,kalantzisDomainDecompositionApproaches2018,doleanNonlinearPreconditioningHow2016}. Our method adopts the inner-outer paradigm, since we already have access to a quasi-optimal eigensolver from~\cite{stammQuasiOptimalFactorizationPreconditioner2022} and use DD for solving the shifted linear systems, revealing an interesting connection between eigenvalue solver preconditioning and linear solver coarse space construction \textendash\ a topic not widely discussed in existing literature.
\par
Coarse spaces for overlapping Schwarz methods are, thus, the last to be mentioned. Spectral coarse spaces, in particular, usually deal with high-contrast diffusion problems by using local cell problems in the volume~\cite{galvisDomainDecompositionPreconditioners2010a,galvisDomainDecompositionPreconditioners2010b,bastianMultilevelSpectralDomain2022}, in the overlap~\cite{spillaneAbstractRobustCoarse2014,spillaneRobustDomainDecomposition2014}, or based on DtN maps~\cite{doleanAnalysisTwolevelSchwarz2012} to construct efficient coarse basis functions based on spectral problems. Recent developments try to extend these approaches to the algebraic case~\cite{gouarinFullyAlgebraicDomain2021,aldaasEfficientAlgebraicTwoLevel2023,aldaasRobustAlgebraicMultilevel2022} and establish the connection to multiscale methods in general~\cite{chupengWavenumberExplicitConvergence2023,heinleinMultiscaleCoarseSpaces2018,heinleinAdaptiveGDSWCoarse2019}. Thus, multiscale methods are naturally related and have been used successively for the GPE~\cite{henningTwoLevelDiscretizationTechniques2014,henningOptimalConvergenceRates2023,peterseimSuperlocalisedWaveFunction2023}.

\subsection{Outline}\label{ssec:outline}
In the following \cref{sec:dd-for-eigen}, we present our two-level algorithm after introducing inner-outer eigenvalue algorithms in \cref{ssec:inner-outer-eigen-algorithms}, their usage with two-level domain decomposition methods in \cref{ssec:two-level-dd}, and the definition of the new coarse space in \cref{ssec:coarse-space}.
After illustrating the connection to spectral coarse spaces, we analyze the coarse space and provide a condition number bound in \cref{sec:analysis}. Finally, \cref{sec:numericalExperiments} presents various numerical examples to test the method with full generality.

\section{Domain decomposition for eigenvalue algorithms}\label{sec:dd-for-eigen}
For the numerical solution of the Schrödinger equation~\eqref{eq:schroedingerEquation}, we first apply a classical Galerkin finite element scheme. Let \(\mathcal{T}_h(\Omega_1)\) be a conforming and shape-regular partition of the unit cell \(\Omega_1\) into finite elements \(\tau \in \mathcal{T}_h(\Omega_1)\) and \(\mathcal{T}_h\) the partition of \(\Omega_L\) using translated copies of \(\mathcal{T}_h(\Omega_1)\), where \(h := \max_{\tau \in \mathcal{T}_h} \operatorname{diam}{\tau}\)~\cite[p150]{brennerMathematicalTheoryFinite1994}. We then define the finite element subspace \(V_h(\Omega_L) \subset H_0^1(\Omega_L), |V_h(\Omega_L)|=n\), that consists of piecewise polynomials with total degree \(r\) from the space of polynomials \(\mathcal{P}_r\) as \(V_h(\Omega_L) := \left\{u \in H_0^1(\Omega_L) \, \middle|\, u|_\tau \in \mathcal{P}_r \, \forall \tau \in \mathcal{T}_h\right\}\). The resulting discrete problem then searches for the function \(\phi_h \in V_h(\Omega_L) \setminus \{0\}\) and the value \(\lambda_h \in \mathbb{R}\) such that
\begin{equation}\label{eq:discrete-weak-form}
  \forall v_h \in V_h(\Omega_L):
  \quad
  \int_{\Omega_L} \nabla \phi_h \cdot \nabla v_h \dd \boldsymbol{z}
  +
  \int_{\Omega_L} V \phi_h v_h \dd \boldsymbol{z}
  =
  \lambda_h
  \int_{\Omega_L} \phi_h v_h \dd \boldsymbol{z}
  .
\end{equation}
For the finite element basis of \(V_h(\Omega_L)\), denoted by \({\{N_i\}}_{i=1}^{n}\), we have \(\phi_h = \sum_{i=1}^n x_i N_i\) where we collect all coefficients in the vector \(\boldsymbol{x}_h \in \mathbb{R}^n\). Thus, \cref{eq:discrete-weak-form} is equivalent to the algebraic generalized eigenvalue problem
\begin{equation}\label{eq-algebraic-eigenvalue-problem}
  \boldsymbol{A} \boldsymbol{x}_h = \lambda_h \boldsymbol{B} \boldsymbol{x}_h
  ,
\end{equation}
in which \(\boldsymbol{A} \in \mathbb{R}^{n \times n}\) is the sum of the usual Laplacian stiffness matrix and the contribution from the potential \(V\), while \(\boldsymbol{B} \in \mathbb{R}^{n \times n}\) is the usual \(L^2\)-mass matrix.
In order to solve the GEVP~\eqref{eq-algebraic-eigenvalue-problem} for the lowest eigenpair \((\boldsymbol{x}_h^{(1)},\lambda_h^{(1)})\), we can employ an inverse power method. This algorithm, starting from an initial guess of \(\boldsymbol{x}_0 \in \mathbb{R}^n\), successively applies the inverse of the operator as \(\boldsymbol{x}_{k+1} = {\boldsymbol{A}}^{-1} \boldsymbol{B} \boldsymbol{x}_{k}\) and \(\boldsymbol{B}\)-normalizes the result. As motivated in \cref{sec:motivation} for the Laplacian EVP and proved in~\cite{stammQuasiOptimalFactorizationPreconditioner2022} for the Schrödinger EVP with \cref{ass:pot-dir-periodic,ass:pot-ess-bounded}, we need to apply shifting to obtain robustness for the domain size \(L\). 
This preconditioning is archived by calculating the quasi-optimal shift \(\sigma\) as the first eigenvalue of the \emph{limit cell problem}~\eqref{eq:cell-problem} posed on the unit cell \(\Omega_1 \subset \Omega_L\): Find \((\psi_h,\sigma) \in (V_{h,\#}(\Omega_1) \setminus \{0\}) \times \mathbb{R}\) such that
\begin{equation}\label{eq:weak-form-cell-problem}
  \forall v_h \in V_{h,\#}(\Omega_1):
  \quad
  \int_{\Omega_1} \nabla \psi_h \cdot \nabla v_h \dd \boldsymbol{z}
  +
  \int_{\Omega_1} V \psi_h v_h \dd \boldsymbol{z}
  =
  \sigma
  \int_{\Omega_1} \psi_h v_h \dd \boldsymbol{z}
  ,
\end{equation}
in which, for any period-aligned box \(\Omega_{\#} \in \mathcal{B}_{\#} := \{\bigtimes_{i=1}^p ( a_i,b_i) \times {(0, \ell)}^{q} \,|\, a_i,b_i \in \mathbb{Z}, 0 \le a_i < b_i \le L\}\), \(V_{h,\#}(\Omega_{\#}) := \big\{u \in H^1(\Omega_{\#}) \, \big|\, u|_\tau \in \mathcal{P}_r \, \forall \tau \in \mathcal{T}_h(\Omega_{\#}), B_{\#}[u]=0 \textnormal{ on } (\partial \Omega_{\#})_{\boldsymbol{x}}, B_{\textnormal{D}}[u]=0 \textnormal{ on } (\partial \Omega_{\#})_{\boldsymbol{y}}\big\}\), with \(\mathcal{T}_h(\Omega_{\#}) = \{\tau \in \mathcal{T}_h \, |\, \tau \subset \bar{\Omega}_{\#}\}\), is the finite element space with periodic condition \(B_{\#}[u](\boldsymbol{z}) = u(\boldsymbol{x},\boldsymbol{y}) - u((x_i - \tilde{n}_i(\boldsymbol{z})(b_i-a_i) )_{i=1}^p, \boldsymbol{y})\) on \((\partial \Omega_{\#})_{\boldsymbol{x}} = \bigtimes_{i=1}^p \{a_i,b_i\} \times (0,\ell)^q\) (with unit normal-vector \(\tilde{\boldsymbol{n}}\)) and zero condition \(B_{\textnormal{D}}[u](\boldsymbol{z}) = u(\boldsymbol{z})\) on \((\partial \Omega_{\#})_{\boldsymbol{y}} = \bigtimes_{i=1}^p (a_i,b_i) \times \{0,\ell\}^q\).  
Let \(\tilde{\boldsymbol{\psi}}\) be the vector of coefficients of \(\psi_h\) and \(\boldsymbol{\psi} \in \mathbb{R}^n\) its periodic extension to \(\Omega_L\). Then, using the shifted operator \(\boldsymbol{A}_{\sigma} = \boldsymbol{A} - \sigma \boldsymbol{B}\) in the IPM results in convergence rates independently of \(L\), as shown in~\cite{stammQuasiOptimalFactorizationPreconditioner2022}.  
\par
Another class of eigensolvers interprets the EVP as an energy minimization problem that tries to minimize the Rayleigh quotient \(R_{\boldsymbol{A},\boldsymbol{B}}(\boldsymbol{x}) := (\boldsymbol{x}^T \boldsymbol{A} \boldsymbol{x})/(\boldsymbol{x}^T \boldsymbol{B} \boldsymbol{x})\) since
\begin{equation}
  \lambda_h^{(1)} = \min_{
    \substack{S \le \mathbb{R}^n \\ |S| = 1}
  }
  \max_{
    \substack{\boldsymbol{x} \in S \\ \boldsymbol{x} \neq \boldsymbol{0}}
  }
  R_{\boldsymbol{A},\boldsymbol{B}}(\boldsymbol{x})
  .
\end{equation}
Taking the gradient, \(\nabla R_{\boldsymbol{A},\boldsymbol{B}}(\boldsymbol{x}) = \tfrac{2}{\boldsymbol{x}^T \boldsymbol{B} \boldsymbol{x}} (\boldsymbol{A}\boldsymbol{x} - R_{\boldsymbol{A},\boldsymbol{B}}(\boldsymbol{x}) \boldsymbol{B} \boldsymbol{x})\), allows us to formulate the steepest descent method as
\begin{equation}
  \boldsymbol{x}_{k+1} = \boldsymbol{x}_{k} - \tau_k \nabla R_{\boldsymbol{A},\boldsymbol{B}}(\boldsymbol{x}_k)
  ,
\end{equation}
where the (adaptive) stepsize \(\tau_k \in \mathbb{R}\) needs to be determined. Moving the scalar factor \(2/(\boldsymbol{x}^T \boldsymbol{B} \boldsymbol{x})\) from \(\nabla R_{\boldsymbol{A},\boldsymbol{B}}\) into \(\tau_k\), one can observe that the descent direction in the \(k\)-th step is proportional to the \emph{spectral residual} \(\boldsymbol{r}_k(\boldsymbol{x}) := \boldsymbol{A}\boldsymbol{x}_k - R_{\boldsymbol{A},\boldsymbol{B}}(\boldsymbol{x}_k) \boldsymbol{B} \boldsymbol{x}_k\). To improve the convergence rate~\cite{baiTemplatesSolutionAlgebraic2000}, we can also apply a preconditioner \(\boldsymbol{P} \in \mathbb{R}^{n \times n}\) and use \(\boldsymbol{w}_k = \boldsymbol{P} \boldsymbol{r}_k\) as the preconditioned search direction. Using the common choice of \(\boldsymbol{P} = \boldsymbol{A}^{-1}\) and \(\tau_k=1\), the iteration \(\boldsymbol{x}_{k+1} = \boldsymbol{x}_{k} - \tau_k \boldsymbol{A}^{-1} \boldsymbol{r}_k\) reduces, up to normalization, back to the IPM\@. One might apply line search methods to find the optimal \(\tau_k\) in each step. However, a common strategy is to use the Rayleigh--Ritz procedure~\cite[p39]{baiTemplatesSolutionAlgebraic2000} within the space \(\tilde{S}_k := \operatorname{span}\{\boldsymbol{w}_k,\boldsymbol{x}_k\}\) by solving a two-dimensional eigenvalue problem. The LOPCG~\cite{knyazevPreconditionedConjugateGradient1991} method adds another previous iterate \(\boldsymbol{x}_{k-1}\) into the search space \(S_k := \operatorname{span}\{\boldsymbol{w}_k,\boldsymbol{x}_{k},\boldsymbol{x}_{k-1}\}\) and finds the element with minimal \(R_{\boldsymbol{A},\boldsymbol{B}}\) per iteration within the whole subspace \(S_k\). As for the IPM, we also use the quasi-optimal shift-and-invert (SI) preconditioner \(\boldsymbol{P} := \boldsymbol{A}_{\sigma}^{-1} = {(\boldsymbol{A} - \sigma \boldsymbol{B})}^{-1}\). This SI-preconditioner will also work for other gradient-based eigenvalue solvers.
\par
However, applying the eigen-preconditioner \(\boldsymbol{P}\) in practice amounts to solving a large and sparse linear system, which requires using iterative, inexact methods when \(n\) is large.

\subsection{Inexact inner-outer eigenvalue algorithms}\label{ssec:inner-outer-eigen-algorithms}
Since we deal with PDE-based problems on meshes, applying the eigenvalue preconditioner for \(\boldsymbol{w}_k = \boldsymbol{A}_{\sigma}^{-1} \boldsymbol{r}\) uses an iterative method. This strategy leads to an \emph{inner-outer}~\cite{freitagInnerouterIterativeMethods2007,freitagConvergenceInexactInverse2007} eigenvalue solver since inner iterations are needed to solve the linear systems for each outer eigenvalue iteration inexactly. As \(\boldsymbol{A}_{\sigma}\) is symmetric positive-definite, the CG method is the preferred method. Defining an outer and an (adaptive) inner tolerance \(0 < \texttt{TOL}_{\textnormal{o}},\texttt{TOL}_{\textnormal{i},k} \in \mathbb{R}\), the \emph{inexact} SI-LOPCG (\cref{alg:si-lopcg}) results. As the algorithm is a generalization of the plain inexact IPM, a sufficiently small constant \(\texttt{TOL}_{\textnormal{i},k}\) ensures convergence~\cite[Cor.~3.2]{freitagConvergenceInexactInverse2007} with the usual assumptions.
Moreover, \(\texttt{TOL}_{\textnormal{i},k}\) can also be relative to the first inner residual norm or depend on the outer residual.
\par
The inner-outer paradigm is related to Newton--Krylov-type methods~\cite{klawonnExtremelyScalableNonlinear2015} for nonlinear equations. Applied to the EVP case, this would lead to the cubically convergent Rayleigh quotient iteration (RQI)~\cite{tapiaInverseShiftedInverse2018}, where convergence to the first eigenpair is not generally guaranteed~\cite[p53]{baiTemplatesSolutionAlgebraic2000}. Therefore, we stick with first-order methods.  
\begin{algorithm}[t]
  \caption{Inexact SI-Preconditioned LOPCG}\label{alg:si-lopcg}
  \begin{algorithmic}[1]
      \Require{%
        \(\boldsymbol{A},\boldsymbol{B} \in \mathbb{R}^{n \times n}\); \(\texttt{TOL}_{\textnormal{i},k},\texttt{TOL}_{\textnormal{o}} > 0\); \(k_{\max} \in \mathbb{N}\); vectors \(\boldsymbol{x}_{0},\boldsymbol{x}_{1} \in \mathbb{R}^n\); shift \(\sigma \in \mathbb{R}\)%
      }
      \Function{%
        SI-LOPCG%
      }{}
      \State{%
        Initialize \(k:=1\), \(\boldsymbol{r}_{\textnormal{o},k} := \boldsymbol{A} \boldsymbol{x}_k - R_{\boldsymbol{A},\boldsymbol{B}}(\boldsymbol{x}_k) \boldsymbol{B} \boldsymbol{x}_k\), \(\boldsymbol{A}_{\sigma} := \boldsymbol{A} - \sigma \boldsymbol{B}\)%
      }
      \While{%
        \({\|\boldsymbol{r}_{\textnormal{o},k}\|}_2 \le \texttt{TOL}_{\textnormal{o}}\) and \(k \le k_{\max}\)%
      }
      \Comment{%
        outer loop%
      }
      \State{%
        Solve \(\boldsymbol{A}_{\sigma} \boldsymbol{w}_k = \boldsymbol{r}_{\textnormal{o},k}\) inexactly s.t.\ \({\|\boldsymbol{A}_{\sigma} \boldsymbol{w}_k - \boldsymbol{r}_{\textnormal{o},k}\|}_2 \le \texttt{TOL}_{\textnormal{i},k}\)%
      }\label{algstep:inner-loop}
      \Comment{%
        inner loop%
      }
      \State{%
        \(S_k := \operatorname{span} \{\boldsymbol{w}_{k},\boldsymbol{x}_{k},\boldsymbol{x}_{k-1}\}\)%
      }
      \State{%
        \(\tilde{\boldsymbol{x}}_{k+1} := \argmin_{\boldsymbol{y} \in S_k \setminus \{\boldsymbol{0}\}} R_{\boldsymbol{A},\boldsymbol{B}}(\boldsymbol{y})\) %
      }
      \Comment{%
        via orthogonalized Rayleigh--Ritz%
      }%
      \State{%
        Normalize \(\boldsymbol{x}_{k+1} := \tilde{\boldsymbol{x}}_{k+1} / (\tilde{\boldsymbol{x}}_{k+1}^T \boldsymbol{B} \tilde{\boldsymbol{x}}_{k+1})^{1/2}\), set \(k := k+1\), update \(\boldsymbol{r}_{\textnormal{o},k}\)
      }
      \EndWhile{}
      \State{%
        \(\lambda_{k} := R_{\boldsymbol{A},\boldsymbol{B}}(\boldsymbol{x}_{k})\)%
      }
      \State{%
        \textbf{return} eigenpair approximation \((\lambda_{k},\boldsymbol{x}_{k})\)%
      }
    \EndFunction{}
  \end{algorithmic}
\end{algorithm}

\subsection{Two-level domain decomposition}\label{ssec:two-level-dd}
In the inner linear system loop (line~\ref{algstep:inner-loop} of \cref{alg:si-lopcg}), a system of the form \(\boldsymbol{A} \boldsymbol{x} = \boldsymbol{b}\) with a sparse matrix \(\boldsymbol{A} \in \mathbb{R}^{n \times n}\) and a given vector \(\boldsymbol{b} \in \mathbb{R}^n\) (\(\boldsymbol{A} \mapsto \boldsymbol{A}_{\sigma}, \boldsymbol{b} \mapsto \boldsymbol{r}_{\textnormal{o},k}\) in our specific case) must be solved. Let \(\mathcal{N}\) be the set of degrees of freedom, i.e., \(|\mathcal{N}|=n\), and define an (overlapping) decomposition into \(N \ll n\) subsets, \(\{\mathcal{N}_i\}_{i=1}^N\), such that \(\mathcal{N} = \cup_{i=1}^N \mathcal{N}_i\).  
This algebraic decomposition represents the geometric domain decomposition of \(\Omega_L\) into subdomains. For all \(\mathcal{N}_i\), there are associated \emph{restriction} matrices \(\boldsymbol{R}_i \in \{0,1\}^{|\mathcal{N}_i| \times |\mathcal{N}|}\), such that \(\boldsymbol{R}_i \boldsymbol{x}\) restricts \(\boldsymbol{x}\) to the subdomain \(\mathcal{N}_i\). The transpose \(\boldsymbol{R}_i^T \in \{0,1\}^{|\mathcal{N}| \times |\mathcal{N}_i|}\) is called the \emph{extension} matrix. For the overlapping case with \(\mathcal{N}_i \cap \mathcal{N}_j \neq \emptyset\) for some \(i,j\), we can also define the diagonal \emph{partition of unity} (PU) matrices \(\boldsymbol{D}_i \in \mathbb{R}^{|\mathcal{N}_i| \times |\mathcal{N}_i|}\), such that \(\boldsymbol{x} = \sum_{i=1}^N \boldsymbol{R}_i^T \boldsymbol{D}_i \boldsymbol{R}_i \boldsymbol{x}\) for all \(\boldsymbol{x}\) holds. An easy choice~\cite[p12]{doleanIntroductionDomainDecomposition2015} is to set \({(\boldsymbol{D}_i)}_{kk} := 1/|\mathcal{M}_k|\) where \(\mathcal{M}_k := \{1 \le i \le N | k \in \mathcal{N}_i\}\) is the set of subdomains that contain the degree of freedom \(k\).  
\par
Then, we can define the one-level \emph{Additive Schwarz}~\cite{dryjaAdditiveVariantSchwarz1987} (AS) and the \emph{restricted Additive Schwarz}~\cite{caiRestrictedAdditiveSchwarz1999} (RAS) preconditioners as
\begin{equation}\label{eq:dd-preconditioners}
  \boldsymbol{M}_{\textnormal{AS},1}^{-1} = \sum_{i=1}^{N} \boldsymbol{R}_i^T \boldsymbol{A}_i^{-1} \boldsymbol{R}_i^{}
  , \quad
  \boldsymbol{M}_{\textnormal{RAS},1}^{-1} = \sum_{i=1}^{N} \boldsymbol{R}_i^T \boldsymbol{D}_i^{} \boldsymbol{A}_i^{-1} \boldsymbol{R}_i^{}
  ,
\end{equation}
where \(\boldsymbol{A}_i := \boldsymbol{R}_i^{} \boldsymbol{A} \boldsymbol{R}_i^T \in \mathbb{R}^{|\mathcal{N}_i| \times |\mathcal{N}_i|}\) are the subdomain coefficient matrices. We can formulate the \emph{stationary RAS method} (i.e., a preconditioned fixed point iteration~\cite[p13]{doleanIntroductionDomainDecomposition2015}) to solve the linear system directly using
\begin{equation}\label{eq:stationary-ras}
  \boldsymbol{x}_{k+1}
  =
  \boldsymbol{x}_k + \boldsymbol{M}^{-1}_{\textnormal{RAS},1} \boldsymbol{r}_{k}
  , \quad
  \boldsymbol{r}_k := \boldsymbol{b} - \boldsymbol{A} \boldsymbol{x}_k
  .
\end{equation}
Using \cref{eq:stationary-ras} for the QOSI-preconditioned inner systems would not result in a scalable method since the condition number of \(\boldsymbol{M}^{-1}_{\textnormal{RAS},1} \boldsymbol{A}_{\sigma}\) will explode for \(L \to \infty\). Thus, we will propose now a two-level modification to the AS/RAS-preconditioners that allows for full \(L\)-robustness. In general, the second level is usually incorporated via a coarse subspace \(S_0 \le \mathbb{R}^n\) with dimension \(0 \le n_0 \ll n\) that has an associated matrix \(\boldsymbol{R}_0^T \in \mathbb{R}^{n \times n_0}\) (which we will specify in \cref{ssec:coarse-space}) such that the columns of \(\boldsymbol{R}_0^T\) are linearly independent and span \(S_0\). We can then define \(\boldsymbol{A}_0 := \boldsymbol{R}_0^{} \boldsymbol{A} \boldsymbol{R}_0^T\) in the usual notation and include this \emph{coarse correction} into both preconditioners from \cref{eq:dd-preconditioners} in an additive fashion as
\begin{equation}\label{eq:two-level-preconditioners}
  \boldsymbol{M}_{\star,2}^{-1} :=  \boldsymbol{R}_0^T \boldsymbol{A}_0^{-1} \boldsymbol{R}_0^{} + \boldsymbol{M}_{\star,1}^{-1}
  ,
\end{equation}
where \(\star\) stands for either AS or RAS\@. The coarse level can also be included in the stationary RAS iteration in a multiplicative fashion~\cite[p302]{ganderCompleteOptimalOptimized2018a}\ as
\begin{equation}\label{eq:two-level-stationary-ras-mult-cc}
  \begin{aligned}
  \boldsymbol{x}_{k+1/2}
  =
  \boldsymbol{x}_k + \boldsymbol{M}_{\textnormal{RAS},1}^{-1} \boldsymbol{r}_k
  ,
  \quad
  \boldsymbol{x}_{k+1}
  =
  \boldsymbol{x}_{k+1/2} + \boldsymbol{R}_0^T \boldsymbol{A}_0^{-1} \boldsymbol{R}_0^{} \boldsymbol{r}_{k+1/2}
  .
  \end{aligned}
\end{equation}
However, using the preconditioners directly in Krylov solvers is beneficial~\cite[p125]{ciaramellaIterativeMethodsPreconditioners2022}, i.e., using the symmetric \(\boldsymbol{M}_{\textnormal{AS},2}^{-1}\) for the CG and the unsymmetric \(\boldsymbol{M}_{\textnormal{RAS,2}}^{-1}\) for the GMRES method. With that description, the coarse space \(S_0\) or its basis as columns of \(\boldsymbol{R}_0^T\) remains to be specified.

\subsection{PerFact: A periodic spectral coarse space based on asymptotic factorization}\label{ssec:coarse-space}
For the classical Poisson problem, the usual motivation for coarse spaces is the lack of global information exchange~\cite[p103]{doleanIntroductionDomainDecomposition2015}. Decomposing a domain \(\Omega_L\) into several subdomains necessarily leads to the emergence of tiny inner subdomains that are very far from the boundary in the subdomain connectivity graph. These inner subdomains then need a lot of DD iterations until the boundary data is propagated.
\par
Although the initial problem~\eqref{eq:schroedingerEquation} with positive \(V\) is spectrally equivalent to Poisson's problem, the situation is different in our setup. The focus is on the geometrical expansion of \(\Omega_L\) with fixed subdomain and discretization size. So, choosing \(N \sim L\) would not increase the distance of subdomains to the boundary and would suggest no need for a coarse space using the classical argument. However, the shifting strategy changes the spectrum of \(\boldsymbol{B}^{-1} \boldsymbol{A}_{\sigma}\) by shifting it closer to the origin. This asymptotic loss of coercivity is the mechanism behind the need for a coarse space in our setup.  
\par
However, the usual rule to include \emph{slow modes} in \(V_0\) still applies, although low-energy modes are the more precise terminology. Moreover, surprisingly, these modes are already at hand. Reconsider the factorization of the ground-state solution \(\phi_L\) of \cref{eq:schroedingerEquation} into the unit cell solution \(\psi\) of \cref{eq:cell-problem} plus a remainder \(u\) as
\begin{equation}\label{eq:factorization-short}
  \phi_L = \psi \cdot u
  ,\quad
  \lambda_L = \sigma + \lambda_{u}
  \textnormal{ where } \lambda_{u} \in \mathcal{O}(1/L^2)
  ,
\end{equation}
that was shown in~\cite[Thm.~2.5]{stammQuasiOptimalFactorizationPreconditioner2022}. This decomposition is a generalization of the separation of variables in the pure Laplacian case. In \cref{eq:factorization-short}, \(\sigma\) is the asymptotic limit of the desired first eigenvalue. Equivalently spoken, the function \(\psi\) corresponding to \(\sigma\) approaches more and more an eigenfunction of the shifted \(\boldsymbol{A}_{\sigma}\) operator with zero eigenvalue \textendash\ a zero energy mode. This observation motivates us to include exactly this function in the coarse space. We define the following with the coefficient vector \(\boldsymbol{\psi} \in \mathbb{R}^n\) from the periodic unit cell problem~\eqref{eq:weak-form-cell-problem}.  
\begin{definition}\label{def:algebraic-perfact-coarse-space}%
  The matrix
  \begin{equation}
    \boldsymbol{R}_0^T :=
    \begin{bmatrix}\label{eq:algebraic-coarse-space}
      \boldsymbol{D}_1 \boldsymbol{R}_1 \boldsymbol{\psi} & 0                         & \cdots & 0                         \\
      0                         & \boldsymbol{D}_2 \boldsymbol{R}_2 \boldsymbol{\psi} & \ddots & \vdots                    \\
      \vdots                    & \ddots                    & \ddots & 0                         \\
      0                         & \cdots                    & 0      & \boldsymbol{D}_N \boldsymbol{R}_N \boldsymbol{\psi} \\
    \end{bmatrix}
    \in \mathbb{R}^{n \times N}
    ,
  \end{equation}
  defines the components of the periodic factorization (PerFact) coarse space.
\end{definition}
\begin{remark}
  The structure of \cref{eq:algebraic-coarse-space} resembles the shape of the classical Nicolaides coarse space~\cite{nicolaidesDeflationConjugateGradients1987} where the local Laplacian kernels, i.e., the constant vectors \(\boldsymbol{1}\), are replaced by their corresponding local kernels \(\boldsymbol{\psi}\) of the shifted Schrödinger operator. While the constant vectors \(\boldsymbol{1}\) might be able to capture global information exchange, we will show that the high-frequency features of \(\boldsymbol{\psi}\) must also be present for the coarse space to work. Also note that \cref{eq:algebraic-coarse-space} technically contains more basis functions than needed, as we will show in \cref{sec:analysis}, since it also includes boundary-touching domains. However, we keep this structure for simplicity and consistency with the literature~\cite[p108]{doleanIntroductionDomainDecomposition2015}.
\end{remark}
In \cref{fig:facprinciple-and-coarsespace}, we display the factorization principle \cref{eq:factorization-short} and the coarse space components \cref{eq:algebraic-coarse-space} for an exemplary decomposition into four subdomains. Therein, the main advantages become visible \textendash\ the vector \(\boldsymbol{\psi}\) is already available from the computation of the QOSI shift (i.e., problem~\cref{eq:cell-problem}) and the small dimension \(n_0 = N\) since every subdomain only needs one basis function in the coarse space in contrast to other adaptive coarse space strategies.
The PerFact coarse space is mainly used in the two-level additive Schwarz preconditioner within the CG method, but an application of the stationary two-level RAS algorithm also works, as we will see in \cref{sec:numericalExperiments}.
\begin{figure}[t]
  \centering%
  \newcommand{\datapath}{./plots/coarse-space-sketch/}%
  \subfloat[%
    Visualization of the factorization principle.%
  ]{%
    \input{plots/coarse-space-sketch/factorization-ann.tex}%
  }%
  \hfill%
  \subfloat[%
    Components of the PerFact coarse space.%
  ]{%
    \input{plots/coarse-space-sketch/components-ann.tex}%
  }%
  \caption{%
    Finite element representation of the \textnormal{\textbf{(a)}} factorization principle from \cref{eq:factorization-short} and the \textnormal{\textbf{(b)}} coarse space basis functions from \cref{eq:algebraic-coarse-space} for the equal-weights partition of unity and an overlap region of two elements between subdomains.%
  }\label{fig:facprinciple-and-coarsespace}%
\end{figure}
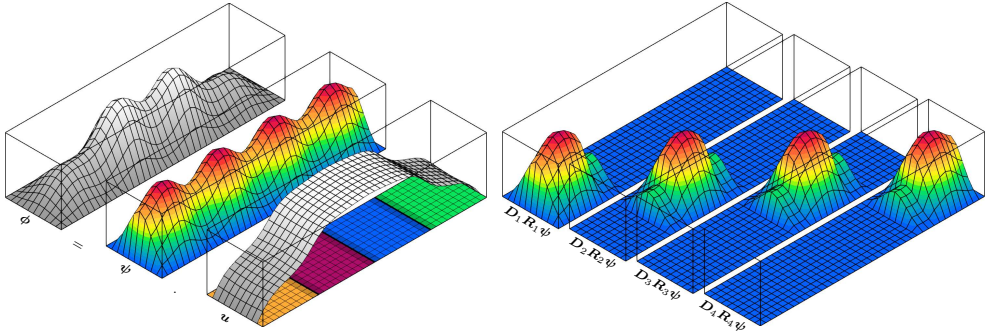%

\section{Analysis of the two-level Additive Schwarz preconditioner}\label{sec:analysis}
In general, using an s.p.d.~preconditioner \(\boldsymbol{M}\) within the CG method to solve the shifted linear systems \(\boldsymbol{A}_\sigma \boldsymbol{x} = \boldsymbol{b}\) leads to the convergence estimate (see, e.g.,~\cite[Lem.~C.10]{toselliDomainDecompositionMethods2005}) of
\begin{equation}\label{eq:cg-energy-norm-error-reduction}
  {\|\boldsymbol{x} - \boldsymbol{x}_k\|}_{\boldsymbol{A}_\sigma} \le 2 {\left(\frac{\sqrt{\kappa(\boldsymbol{M}^{-1}\boldsymbol{A}_\sigma)}-1}{\sqrt{\kappa(\boldsymbol{M}^{-1}\boldsymbol{A}_\sigma)}+1}\right)}^k {\|\boldsymbol{x} - \boldsymbol{x}_0\|}_{\boldsymbol{A}_\sigma}
  .
\end{equation}
For increasing domain sizes \(L\), applying the quasi-optimal shift-and-invert technique from~\cite{stammQuasiOptimalFactorizationPreconditioner2022} results in \(\lambda_{\min}(\boldsymbol{A}_\sigma) \to 0\) as \(L \to \infty\). However, since the domain expands anisotropically with constant \(\ell\), we have \(\lambda_{\max}(\boldsymbol{A}_\sigma) \ge C > 0\) for all \(L \to \infty\) by using the Laplacian EVP as a lower bound and following the argumentation of \cref{sec:motivation}.
\par
The behavior of the spectrum thus results in \(\kappa(\boldsymbol{A}_\sigma) \to \infty\) for \(L \to \infty\). Therefore, without a preconditioner, i.e., \(\boldsymbol{M}=\boldsymbol{I}\), the number of iterations to decrease the relative error below a given tolerance will drastically increase for \(L \to \infty\) in the linear solution step of \cref{alg:si-lopcg} (line~\ref{algstep:inner-loop}). However, the two-level AS preconditioner \(\boldsymbol{M}_{\textnormal{AS},2}^{-1}\) from \cref{eq:two-level-preconditioners} is \(L\)-robust, i.e., \(\kappa(\boldsymbol{M}_{\textnormal{AS},2}^{-1} \boldsymbol{A}_\sigma) \le C\) for all \(L\), since we can control the lower part of the spectrum with the PerFact coarse space, as shown in the following.

\subsection{Abstract theory}\label{ssec:abstract-theory}
After the conforming finite element discretization with \(\overline{\Omega}_L = \bigcup_{\tau \in \mathcal{T}_h} \tau\) from \cref{sec:dd-for-eigen}, we operate in the finite-dimensional setting with \(V_h(\Omega_L) \subset H^1_0(\Omega_L)\), where we now use the abbreviation \(V_h := V_h(\Omega_L)\) for simplicity. In each step of the eigenvalue algorithm, we have to apply the preconditioner, which amounts to solving a source problem: Given \(f_h \in V_h'\), find \(u_h \in V_h\) such that
\begin{equation}\label{eq:stepLinearSystem}
  a_\sigma(u_h,v_h) = \langle f_h, v_h \rangle ,\quad \forall v_h \in V_h
  .
\end{equation}
The bilinear form \(a_\sigma : V_h \times V_h \to \mathbb{R}\) in \cref{eq:stepLinearSystem} corresponds to the shift-and-invert eigenvalue preconditioner matrix \(\boldsymbol{A}_\sigma\) and is given by
\begin{align}\label{eq:shifted_bilinear_form}
  a_{\sigma}(u,v)
  & := {(\nabla u, \nabla v)}_{L^2(\Omega_L)} + {(V u, v)}_{L^2(\Omega_L)} - \sigma {(u, v)}_{L^2(\Omega_L)}
   .
\end{align}
In order to apply a two-level additive Schwarz method, we decompose the domain \(\Omega_L\) into a set of non-overlapping subdomains \(\{\Omega_i'\}_{i=1}^N\) either by using the natural decomposition into \(N=L^p\) repeating unit cells or with the help of an automatic graph partitioning software, e.g., METIS~\cite{karypisMETISSoftwarePackage1997} or Scotch~\cite{chevalierPTScotchToolEfficient2008}. Then, let \(\{\Omega_i\}_{i=1}^L\) with \(\Omega_i' \subset \Omega_i\) be an overlapping domain decomposition by adding \(\delta \ge 1\) layers of elements \(\tau \in \mathcal{T}_h\) based on the element connectivity graph. Following the abstract theory, c.f., e.g.,~\cite{doleanIntroductionDomainDecomposition2015}, we can define for \(i=1,\dots,N\) the space of restrictions of functions \(v \in V_h\) to a subdomain \(\Omega_i\) as \(V_h(\Omega_i) := \{v|_{\Omega_i} \mid v \in V_h\}\) and the space of \(\overline{\Omega}_i\)-supported functions as \(V_{h,0}(\Omega_i) := \{v|_{\Omega_i} \mid v \in V_h, \; \operatorname{supp}(v) \subset \overline{\Omega}_i\}\). Then we have the restriction operators \(r_i : V_h \to V_{h}(\Omega_i), r_i v = v|_{\Omega_i}\) and their adjoints, the extension-by-zero operators, \(r_i^T : V_{h,0}(\Omega_i) \to V_h\), which extends functions by zero outside of \(\Omega_i\). For simplicity, we will sometimes interpret \(V_{h,0}(\Omega_i)\) as a subspace of \(V_h\) leaving out the extension-by-zero operators.  
\par
For the two-level approach, we also define a coarse space \(V_0 \subset V_h\) with corresponding natural embedding (i.e., the inclusion map that is the linear interpolation~\cite{larsonFiniteElementMethod2013} in the present setup) \(r_0^T : V_0 \to V_h\) and its adjoint \(r_0 : V_h \to V_0\). The matrix form of the two-level additive Schwarz preconditioner for \(\boldsymbol{A}_\sigma\) then reads
\begin{equation}\label{eq:twoLevelASMPreconditioner}
  \boldsymbol{M}_{\textnormal{AS},2}^{-1} = \boldsymbol{R}_0^T \boldsymbol{A}_{\sigma,0}^{-1} \boldsymbol{R}_0^{} + \sum_{i=1}^{N} \boldsymbol{R}_i^T \boldsymbol{A}_{\sigma,i}^{-1} \boldsymbol{R}_i^{}
  \textnormal{ with }
  \boldsymbol{A}_{\sigma,i}^{} =  \boldsymbol{R}_i^{} \boldsymbol{A}_{\sigma} \boldsymbol{R}_i^T
  \textnormal{ and }
  \boldsymbol{A}_{\sigma,0}^{} = \boldsymbol{R}_0^{} \boldsymbol{A}_{\sigma} \boldsymbol{R}_0^T
  ,
\end{equation}
with \({\{\boldsymbol{R}_i^{}\}}_{i=1}^N\) and \(\boldsymbol{R}_0\) denoting the matrix representations of \({\{r_i\}}_{i=1}^N\) and \(r_0\) for the given finite element basis. The local matrices \(\boldsymbol{A}_{\sigma,i}\) are invertible since the corresponding restricted bilinear forms \(a_{\sigma,\Omega_i} : V_{h,0}(\Omega_i) \times V_{h,0}(\Omega_i)\) are positive definite as \(a_{\sigma,\Omega_i}(u,v) := a_{\sigma}(r_i^T u, r_i^T v) > 0\) for all \(u,v \in V_{h,0}(\Omega_i)\) and finite \(L\).
\par
Following~\cite{doleanIntroductionDomainDecomposition2015,spillaneAbstractRobustCoarse2014}, we now recall the main ingredients of the abstract additive Schwarz theory. First, we have the following geometric definition.
\begin{definition}[{finite coloring~\cite[Def.~3.1]{bastianMultilevelSpectralDomain2022}}]\label{def:finite-coloring}%
  The partition \(\{\Omega_i'\}_{i=1}^N\) admits a finite coloring with \(N_{\textnormal{c}} \in \mathbb{N}\) colors, \(N_{\textnormal{c}} \le N\), if there exists a map \(c : \{1,\dots,N\} \to \{1,\dots,N_{\textnormal{c}}\}\) such that  
  \begin{equation}\label{eq:coloring-map}
    i \neq j \land c(i)=c(j) \Rightarrow a_{\sigma}(r_i^T v_i, r_j^T v_j) = 0
    ,\quad \forall v_i \in V_{h,0}(\Omega_i), v_j \in V_{h,0}(\Omega_j)
    .
  \end{equation}  
\end{definition}
In practice, a low \(N_{\textnormal{c}}\) is naturally given when aligning the decomposition with the period of \(V\). The following notion is the critical ingredient to show the effectiveness of a coarse space by lower bounding the smallest eigenvalues of the preconditioned system.
\begin{definition}[{stable decomposition~\cite[Def.~2.7]{spillaneAbstractRobustCoarse2014}~\cite[Def.~3.3]{bastianMultilevelSpectralDomain2022}}]\label{def:stable-decomp}%
  Given a coarse space \(V_0 \subset V_h\), local subspaces \({\{V_{h,0}(\Omega_i)\}}_{1 \le i \le N}\) of \(V_h\), and a constant \(C_0 > 0\), a stable decomposition of \(v \in V_h\) is a family of functions \({\{v_i\}}_{0 \le i \le N}\) that satisfy \(v = v_0 + \sum_{i=1}^N v_i\), \(v_i \in V_{h,0}(\Omega_i)\), \(v_0 \in V_0\), such that
  \begin{equation}
    a_{\sigma}(v_0,v_0) + \sum_{i=1}^N a_{\sigma}(v_i,v_i) \le C_0 \, a_{\sigma}(v,v)
    .
  \end{equation}
\end{definition}
Thus, according to~\cite[Cor.~5.12]{doleanIntroductionDomainDecomposition2015}, we aim to find a stable decomposition for all \(v \in V_h\) since it directly yields a condition number bound of \(\boldsymbol{M}_{\textnormal{AS},2}^{-1} \boldsymbol{A}_{\sigma}\). To apply this abstract theory, we need further notation.
\begin{lemma}[{partition of unity}]\label{def:partitionOfUnity}%
  For the overlapping decomposition \(\{\Omega_i\}_{i=1}^N\) of \(\Omega_L\), there exists a family of partition of unity functions \(\{\chi_i\}_{i=1}^N \subset V_h\), such that \(0 \le \chi_i(\boldsymbol{z}) \le 1 \textnormal{ for } \boldsymbol{z} \in \overline{\Omega}_L\), \(\operatorname{supp}(\chi_i) \subset \overline{\Omega}_i\), and \(\sum_{i=1}^N \chi_i(\boldsymbol{z}) = 1\) for all \(\boldsymbol{z} \in \overline{\Omega}_L\). 
\end{lemma}
\begin{proof}
  See, e.g.,~\cite[Lem.~5.7]{doleanIntroductionDomainDecomposition2015} and~\cite[Lem.~3.4]{toselliDomainDecompositionMethods2005} with \(\{\Omega_i\}_{i=1}^N\) satisfying the overlap and finite coloring assumptions. 
\end{proof}
\begin{remark}
  For the construction of the partition of unity, we use the finite element basis representation \(v = \sum_{k=1}^n v_k N_k\) for any \(v \in V_h\) (see \cref{sec:dd-for-eigen}). With \(\operatorname{dof}(\Omega_i)\) denoting the set of all internal degrees of freedoms for subdomain \(\Omega_i\), the most straightforward choice is the equal-weights partition of unity~\cite{spillaneAbstractRobustCoarse2014} where \(\chi_i(v) := \sum_{k \in \operatorname{dof}(\Omega_i)} \frac{1}{\mu_{k}} v_k N_k\) with \(\mu_{k} := |\left\{j \ \middle|\ 1 \le j \le N, k \in \operatorname{dof}(\Omega_j)\right\}|\) denoting the number of subdomains for which \(k\) is an internal degree of freedom. Other choices are based on distance functions, e.g.~\cite[Lem.~3.4]{toselliDomainDecompositionMethods2005}, and have the additional property that \({\|\nabla \chi_i\|}_{\infty} \le C / \delta_i\), where \(\delta_i\) is the minimal overlap thickness of the \(i\)-th subdomain. This distance-based partition of unity is favorable when the overlap varies.
\end{remark}
In calculating the quasi-optimal shift \(\sigma\), we already solved the eigenvalue problem on the unit cell \(\Omega_1\) with \(\boldsymbol{x}\)-periodic and \(\boldsymbol{y}\)-zero boundary conditions in \cref{eq:weak-form-cell-problem}, where we abbreviate its solution \(\psi_h\) from now on as \(\psi\) for simplicity.
We also define \(E_{\boldsymbol{x}}^{\#} : V_{h,\#}(\Omega_1) \to V_{h,\#}(\Omega_L)\), \(E_{\boldsymbol{x}}^{\#} u(\boldsymbol{x},\boldsymbol{y}) := u(\operatorname{mod}(\boldsymbol{x},1), \boldsymbol{y})\), as the periodic extension operator in the \(\boldsymbol{x}\)-direction.
Although we do not solve local eigenvalue problems in practice since \(\psi\) is already computed, we can still analyze the method in the context of spectral coarse spaces when we take care of the periodicity by modifying the decomposition.

\subsection{Aligning the decomposition}\label{ssec:aligning-the-decomposition}
Recall from \cref{eq:geometry-description} that the domain \(\Omega_L\) is an abstract box. Naturally, \(\Omega_L\) is split into \(L^p\) cells, so every cell can be identified with an integer vector \(\boldsymbol{i} \in {\{1,\dots,L\}}^p\) when considering the distance to the unit cell. We map this vector translation to a one-dimensional index via the bijective map \(n : {\{1,\dots,L\}}^p \to \{1,\dots,L^p\}\) where \(n(\boldsymbol{i}) := 1 + \sum_{j=1}^p L^{j-1} (\boldsymbol{i}_j - 1)\).
The corresponding inverse map \(\boldsymbol{n}^{-1} : \{1,\dots,L^p\} \to {\{1,\dots,L\}}^p\) can then be used to define the \textit{periodic decomposition} \(\{\Omega_i^{\#}\}^{L^p}_{i=1}\), where each periodic cell is given by 
\begin{equation}\label{eq:periodic-cells-definition}
  \Omega_i^{\#} :=
  \bigtimes_{j=1}^p \left( {[\boldsymbol{n}^{-1}(i)]}_j-1,{[\boldsymbol{n}^{-1}(i)]}_j \right) \times {(0, \ell)}^{q}
  .
\end{equation}
Every subdomain \(\Omega_i\) belongs to a periodic neighborhood, defined by the following.
\begin{definition}[periodic neighborhood]\label{def:periodicNeighborhood}%
  Let \(\Omega_i \in \{\Omega_i\}^{N}_{i=1}\) be any subdomain. Then, the periodic neighborhood \(\tilde{\Omega}_i\) of \(\Omega_i\) is defined as the smallest period-aligned box (see \cref{sec:dd-for-eigen}) containing \(\Omega_i\), i.e., \ \(\tilde{\Omega}_i := \argmin_{\{\Omega \in \mathcal{B}_{\#} \,|\, \Omega_i \subset \Omega\}} \operatorname{vol}(\Omega)\). With \(\mathcal{I}_i^{\#,\textnormal{n}} := \{ j \in \{1,\dots,L^p\} \,|\, \tilde{\Omega}_i \cap \Omega_j^{\#} \neq \emptyset \}\), we have \(\tilde{\Omega}_i = \bigcup_{i \in \mathcal{I}_i^{\#,\textnormal{n}}} \Omega_i^{\#}\). 
\end{definition}
\cref{fig:geometricDecompositions} visualizes the relation between \(\Omega_i\), \(\Omega_i'\), and \(\tilde{\Omega}_i\). A collection of periodic neighborhoods \(\{\tilde{\Omega}_i\}_{i=1}^N\) is itself again an overlapping decomposition of \(\Omega_L\), where the decomposition reflects the periodicity of the potential \(V\).  
The periodic neighborhood decomposition also induces their corresponding restrictions \(\tilde{r}_i : V_h \to V_h(\tilde{\Omega}_i)\) and partitions of unity \(\tilde{\chi}_i\), defined by the zero-extension of \(\chi_i\) from \(\Omega_i\) to \(\tilde{\Omega}_i\). The set \(\{\tilde{\Omega}_i\}_{i=1}^N\) gives rise to a multiplicity constant in the following.  
\begin{definition}[periodic neighborhood intersection multiplicity]\label{def:periodic-neighborhood-intersection}%
  Let \(\tilde{k}_0\) be the maximum number of periodic neighborhoods to which one periodic cell from \(\{\Omega^{\#}_i\}_{i=1}^{L^p}\) can belong, i.e., \(\tilde{k}_0 := \max_{i \in \{1,\dots,N\}} |\{ j \in \{1,\dots,L^p\} \ |\ \tilde{\Omega}_i \cap \Omega_j^{\#} \neq \emptyset \}|\).  
\end{definition}
\begin{figure}[t]%
  \centering%
  \newcommand{\datapath}{./plots/dd-geometry-sketch-tikz}%
  \subfloat[%
    \(\delta=1\).%
  ]{%
    \input{plots/dd-geometry-sketch-tikz/dd-geometry-sketch-tikz-1.tex}%
  }%
  \hfill%
  \subfloat[%
    \(\delta=2\).%
  ]{%
    \input{plots/dd-geometry-sketch-tikz/dd-geometry-sketch-tikz-2.tex}%
  }%
  \caption{%
    Sketch of the non-overlapping \(\{\Omega_i'\}_{i=1}^8\) (black border), overlapping \(\{\Omega_i\}_{i=1}^8\) (white border), and periodic neighborhood decomposition \(\{\tilde{\Omega}_i\}_{i=1}^8\) (cross-hatch) of \(\Omega_4 := (0,4) \times (0,2)\) for an overlap (dark shades) of \textnormal{\textbf{(a)}} \(\delta=1\) and \textnormal{\textbf{(b)}} \(\delta=2\) layers of elements. An increase of the periodic neighborhood from \(\tilde{\Omega}_1 = (0,2) \times (0,2)\) for \(\delta=1\) to \(\tilde{\Omega}_1 = (0,3) \times (0,2)\) for \(\delta=2\) can be observed. 
  }\label{fig:geometricDecompositions}%
\end{figure}
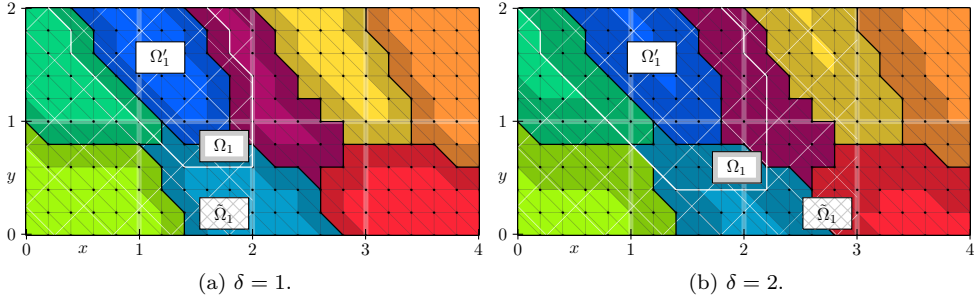%
\par
Similar to the classical Nicolaides coarse space for the Laplace problem~\cite{doleanAnalysisTwolevelSchwarz2012}, we only include subdomains away from the boundary in the coarse space due to technical reasons, such as non-matching boundary conditions. In the present setup, only the boundaries in the expanding \(\boldsymbol{x}\)-direction need particular focus. Thus, let us define the set of all \(\boldsymbol{x}\)-boundary subdomains as \(\mathcal{I}^{\textnormal{b}} := \{i \in 1,\dots,N \mid \partial \Omega_i \cap ({\{0,L\}}^p \times {(0,\ell)}^q) \neq \emptyset\}\). This induces the set of all interior subdomains as \(\mathcal{I}^{\textnormal{i}} := \{1,\dots,N\} \setminus \mathcal{I}^{\textnormal{b}}\). Let us also define the set \(\tilde{\mathcal{I}}^{\textnormal{b}} := \{i \in 1,\dots,N \mid \partial \tilde{\Omega}_i \cap ({\{0,L\}}^p \times {(0,\ell)}^q) \neq \emptyset\}\), whose periodic neighborhoods touch the \(\boldsymbol{x}\)-boundary. The other periodic neighborhoods are grouped into \(\tilde{\mathcal{I}}^{\textnormal{i}} := \{1,\dots,N\} \setminus \tilde{\mathcal{I}}^{\textnormal{b}}\). 
Based on this period-aligned splitting, we define the periodic coarse space.
\begin{definition}[PerFact coarse space]\label{def:periodicCoarseSpace}%
  With \(\psi\) from \cref{eq:weak-form-cell-problem}, the space \(V_0 := \operatorname{span}{\left( {\{\chi_i E^\#_{\boldsymbol{x}}\psi\}}_{i \in \tilde{\mathcal{I}}^{\textnormal{i}}} \right)}\) is called the periodic factorization (PerFact) coarse space.
\end{definition}
\begin{remark}
  Note that \cref{def:periodicCoarseSpace,def:algebraic-perfact-coarse-space} only differ for \(\boldsymbol{x}\)-boundary touching subdomains, which are \(L\)-asymptotically irrelevant.
\end{remark}
For the analysis, we also need the notion of cell symmetry.
\begin{definition}[\(\boldsymbol{x}\)-cell-symmetric mesh]%
  Let \(\mathcal{T}_h\) be a triangulation, \(\bar{x}_i\) the \(i\)-th component of its center of mass \(\bar{\boldsymbol{x}}\), and \(R_{\bar{x}_i}\) the reflection across the plane \(P_{\bar{x}_i} = \left\{ \boldsymbol{z} \in \mathbb{R}^d \ \middle|\ \langle e_i , \boldsymbol{z} \rangle = \bar{x}_i \right\}\). We say \(\mathcal{T}_h\) is \(\boldsymbol{x}\)-symmetric if for all expanding directions with \(i \in \{1,\dots,p\}\), \(\tau \in \mathcal{T}_h \Rightarrow R_{\bar{x}_i} \tau \in \mathcal{T}_h\).
\end{definition}
\begin{definition}[\(\boldsymbol{x}\)-cell-symmetric potential]%
  A potential \(V\) is said to be \(\boldsymbol{x}\)-cell-symmetric for an \(\tilde{\Omega}_i\) if for all \(x_i\)-dimensions with \(i \in \{1,\dots,p\}\)
  \begin{equation}
    V(x_1,\dots,x_i,\dots,x_p,\boldsymbol{y})
    =
    V(x_1,\dots,2\bar{x}_i - x_i,\dots,x_p,\boldsymbol{y})
    \quad \textnormal{ for a.e.~} (\boldsymbol{x},\boldsymbol{y}) \in \tilde{\Omega}_i
    .
  \end{equation}
\end{definition}

\subsection{A condition number bound for cell-symmetric potentials}\label{ssec:condition-number-bound}
We are now prepared to analyze the condition number of the preconditioned system. As we operate near the \emph{edge of coercivity}, meaning that we shifted the operator by \(\sigma\), which is only a little less than the smallest eigenvalue, we must be cautious with the analysis. In particular, we must ensure that local Neumann problems remain positive semidefinite for the theory of spectral coarse spaces to hold. Thus, we base the analysis on the following assumptions.
\begin{assumptionslist}
  \item\label{ass:symm-trian} The triangulation of all periodic neighborhoods \(\tilde{\Omega}_i\) is \(\boldsymbol{x}\)-cell-symmetric.
  \item\label{ass:symm-pot} The potential \(V\) is \(\boldsymbol{x}\)-cell-symmetric for all \(\tilde{\Omega}_i\).
\end{assumptionslist}
\par
Although these assumptions seem restricted, the method still works in the general case, as shown numerically in \cref{sec:numericalExperiments}. To apply the theory of spectral coarse spaces, we show that \(E^\#_{\boldsymbol{x}} \psi\), defined by \cref{eq:weak-form-cell-problem}, is a solution to a generalized eigenvalue problem (GEVP) of a particular structure. In the usual setting, the GEVPs are formulated on overlapping subdomains \(\Omega_i\) or inside the overlapping zone. In our setting, however, the components \(E_{\boldsymbol{x}}^{\#} \psi\) from \cref{def:periodicCoarseSpace} are periodic. This property can be used if we formulate the GEVPs on the periodic neighborhood \(\tilde{\Omega}_i\).
\begin{definition}[GEVPs]\label{def:gevps}%
  For all periodic neighborhoods \(\tilde{\Omega}_i\), we define the following generalized eigenvalue problems on \(W_i\): Find \((\tilde{p}_i^{(k)}, \tilde{\lambda}_i^{(k)}) \in (W_i \setminus \{0\}) \times \mathbb{R}\), such that
  \begin{equation}\label{eq:GEVPEquation}
    \tilde{a}_{\sigma,\tilde{\Omega}_i}(\tilde{p}_i^{(k)},v)
    =
    \tilde{\lambda}_i^{(k)} \tilde{b}_{\tilde{\Omega}_i} (\tilde{p}_i^{(k)},v)
    ,\quad \forall v \in W_i
    ,
  \end{equation}
  where we set \(\tilde{b}_{\tilde{\Omega}_i} (u,v) := \tilde{a}_{0,\tilde{\Omega}_i}(\tilde{\chi}_i u, \tilde{\chi}_i v)\). The default case of \(W_i=V_h(\tilde{\Omega}_i)\) is called the Neumann GEVP (NGEVP), and the case of \(W_i=V_{h,\#}(\tilde{\Omega}_i)\) is called the periodic GEVP (PGEVP), whose eigenpair is then denoted by \((\tilde{p}_{i,\#}^{(k)},\tilde{\lambda}_{i,\#}^{(k)})\).
\end{definition}
\begin{remark}
  There is some flexibility in the choice of \(\tilde{b}_{\tilde{\Omega}_i}\) in \cref{eq:GEVPEquation} (see~\cite{bastianMultilevelSpectralDomain2022,bootlandOverlappingSchwarzMethods2023}) as long as we can extract bounds on \(\|\tilde{\chi}_i \cdot\|_{\tilde{a}_{\sigma,\tilde{\Omega}_i}}\) (for use in \cref{thm:stable-splitting}), rendering, e.g., the plain \(L^2\)-inner product impractical. However, we are only interested in the kernel of \(\tilde{a}_{\sigma,\tilde{\Omega}_i}\), which is unaffected by the specific choice of \(\tilde{b}_{\tilde{\Omega}_i}\) as long as the GEVPs are non-defective. This provides some freedom in carrying out the analysis.
\end{remark}
\begin{remark}
  In contrast to the classical approaches, we do not use the same generating bilinear form to define the left- and the right-hand side of the GEVP\@. In particular, only the left side of \cref{eq:GEVPEquation} has a negative shift term in \(\tilde{a}_{\sigma,\tilde{\Omega}_i}\), while the right-hand side is generated without shift by \(\tilde{a}_{0,\tilde{\Omega}_i}\). This ensures that \(\tilde{b}_{\tau}(u|_\tau,u|_\tau) \ge 0 \; \forall u \in V_h(\tilde{\Omega}_i)\), which guarantees the positive semidefiniteness of \(\tilde{b}_{\tilde{\Omega}_i}\), which is a crucial requirement for the theory. A similar idea finds application in the context of non-self-adjoint or indefinite problems, e.g., in the \(\mathcal{H}\)-GenEO approach of~\cite{bootlandOverlappingSchwarzMethods2023}.
\end{remark}
We will now show one of the crucial observations: the PGEVP and the NGEVP have the same first eigenpair under \cref{ass:symm-trian,ass:symm-pot}. To formalize that result, we need the following notion of reflection.
\begin{definition}[reflection operator]\label{def:reflectionOfFunctions}%
  Let \(\tilde{\Omega}_i\) be an \(\boldsymbol{x}\)-symmetric triangulation. For any function \(v \in V_h(\tilde{\Omega}_i)\) with \(m = |V_h(\tilde{\Omega}_i)|\) DOFs, we define the reflection \(R_{\bar{x}_i} : V_h(\tilde{\Omega}_i) \to V_h(\tilde{\Omega}_i)\) across the \(P_{\bar{x}_i}\)-plane using the finite element basis representation \(v = \sum_{k=1}^m v_k N_k\) as
  \begin{equation}
    v \mapsto R_{\bar{x}_i} v := \sum_{k=1}^m \left( \sum_{l=1}^m R_{kl} v_l \right) N_k
    ,
  \end{equation}
  where the permutation matrix \(\boldsymbol{R} := {\left(R_{kl}\right)}_{1\le k,l \le m}\) (after reordering) has the form
  \begin{equation}
    \boldsymbol{R} = \begin{bmatrix}
      \boldsymbol{I}_{m_1} \otimes \boldsymbol{S} & \boldsymbol{0} \\
      \boldsymbol{0} & \boldsymbol{I}_{m_2}
    \end{bmatrix}
    ,\textnormal{ with the pairwise swap matrix }
    \boldsymbol{S} = \begin{bmatrix}
      0 & 1 \\
      1 & 0
    \end{bmatrix}
    ,
  \end{equation}
  where all \(m\) nodes are split into \(m_2\) on the reflection plane and \(2m_1\) remaining nodes.
\end{definition}
\begin{lemma}\label{lem:pgevp_ngecp_same_groundstate}%
  Assume \labelcref{ass:symm-trian},\labelcref{ass:symm-pot}, and \(i \in \tilde{\mathcal{I}}^{\textnormal{i}}\), then the unique lowest eigenpair of the NGEVPs and PGEVPs from \cref{def:gevps} is given by \((\tilde{p}_i^{(1)}, \tilde{\lambda}_i^{(1)}) = (\tilde{r}_i E_{\boldsymbol{x}}^{\#} \psi,0)\).
\end{lemma}
\begin{proof}
  Let \(i \in \tilde{\mathcal{I}}^{\textnormal{i}}\). Assume that \(p^{(1)}\) is the first eigenfunction of the homogeneous EVP with the standard \(L^2(V_{h}(\tilde{\Omega}_i))\)-inner product, meaning
  \begin{equation}
    p^{(1)}
    \in
    \argmin_{u \in V_{h}(\tilde{\Omega}_i)} \frac{\tilde{a}_{\sigma,\tilde{\Omega}_i}(u,u)}{{(u,u)}_{L^2(V_{h}(\tilde{\Omega}_i))}}
    .
  \end{equation}
  By~\cite[Thm.~8.38]{gilbargEllipticPartialDifferential2001}, \(p^{(1)}\) is simple and strictly positive (by fixing the sign). Now fix an arbitrary \(x_j\)-dimension with \(j \in \{1,\dots,p\}\) and consider the reflection \(q^{(1)} := R_{\bar{x}_j} p^{(1)}\) of \(p^{(1)}\) across the plane \(P_{\bar{x}_j}\). Regarding the basis coefficient vector \(\boldsymbol{p}\) of \(p^{(1)}\), this means \(\boldsymbol{q} = \boldsymbol{R} \boldsymbol{p}\) in the sense of \cref{def:reflectionOfFunctions}. Now, two things can happen.
  \paragraph{Case 1: \(\boldsymbol{q} \in \operatorname{span} \boldsymbol{p}\)} In this case, there exists an \(\alpha \in \mathbb{R}\) such that \(\boldsymbol{p}\) is an eigenvector with \(\boldsymbol{q} := \boldsymbol{R} \boldsymbol{p} = \alpha \boldsymbol{p}\). Since the reflection matrix \(\boldsymbol{R}\) is a permutation matrix and thus an orthogonal matrix, it only has eigenvalues \(\alpha \in \{-1,1\}\), such that there are again two cases. For \(\alpha=1\), we have that \(p^{(1)}\) is symmetric w.r.t.\ the plane \(P_{\bar{x}_j}\) since it is invariant under the reflection. The other case of \(\alpha=-1\) is impossible since all coefficients in \(\boldsymbol{p}\) are strictly positive, and a permutation matrix \(\boldsymbol{R}\) can not change this property for the equality \(\boldsymbol{p} = - \boldsymbol{R} \boldsymbol{p}\) to hold.
  \paragraph{Case 2: \(\boldsymbol{q} \notin \operatorname{span} \boldsymbol{p}\)} In that case, we observe that \(q^{(1)}\) and \(p^{(1)}\) have the same Rayleigh quotient since
  \begin{equation}\label{eq:sameRayleighs}
    \frac{\tilde{a}_{\sigma,\tilde{\Omega}_i}(q^{(1)},q^{(1)})}{{(q^{(1)},q^{(1)})}_{L^2(V_{h}(\tilde{\Omega}_i))}}
    =
    \frac{\tilde{a}_{\sigma,\tilde{\Omega}_i}(R_{\bar{x}_j} p^{(1)},R_{\bar{x}_j} p^{(1)})}{{(R_{\bar{x}_j} p^{(1)},R_{\bar{x}_j} p^{(1)})}_{L^2(V_{h}(\tilde{\Omega}_i))}}
    =
    \frac{\tilde{a}_{\sigma,\tilde{\Omega}_i}(p^{(1)},p^{(1)})}{{(p^{(1)},p^{(1)})}_{L^2(V_{h}(\tilde{\Omega}_i))}}
    ,
  \end{equation}
  in which we used \({\|R_{\bar{x}_j} p^{(1)}\|}^2_{L^2(\tilde{\Omega}_i)} = {\|p^{(1)}\|}^2_{L^2(\tilde{\Omega}_i)}\) by the orthogonality of the reflection \(R_{\bar{x}_j}\) and \(\tilde{a}_{\sigma,\tilde{\Omega}_i}(R_{\bar{x}_j} p^{(1)},R_{\bar{x}_j} p^{(1)}) = \tilde{a}_{\sigma,\tilde{\Omega}_i}(p^{(1)},p^{(1)})\) by the symmetry of the mesh \(\tilde{\Omega}_i\) and the potential \(V\) w.r.t.\ to the \(P_{\bar{x}_j}\)-plane, \labelcref{ass:symm-pot,ass:symm-trian}. However, \cref{eq:sameRayleighs} would imply that the first eigenspace is two-dimensional since \(\boldsymbol{q}\) is not in the \(\operatorname{span} \boldsymbol{p}\), although it has a minimal Rayleigh quotient, which is a contradiction.
  \par
  Thus, all in all, \(p^{(1)}\) must be symmetric w.r.t.\ the \(P_{\bar{x}_j}\)-plane. Since the arguments apply for all expanding directions, the first eigenfunction \(p^{(1)}\) is \(\boldsymbol{x}\)-symmetric, i.e., \(p^{(1)} = R_{\bar{x}_j} p^{(1)}\) for all \(j \in \{1,\dots,p\}\), implying that \(\tilde{b}_{\tilde{\Omega}_i}(R_{\bar{x}_j} p^{(1)},R_{\bar{x}_j} p^{(1)}) = \tilde{b}_{\tilde{\Omega}_i}(p^{(1)},p^{(1)})\). Symmetry also yields periodicity on the \(\boldsymbol{x}\)-boundaries resulting in \(p^{(1)} \in V_{h,\#}(\tilde{\Omega}_i)\). Using the min-max principle for the GEVPs, we thus have
  \begin{equation}
    \min_{u \in V_{h}(\tilde{\Omega}_i)} \frac{\tilde{a}_{\sigma,\tilde{\Omega}_i}(u,u)}{\tilde{b}_{\tilde{\Omega}_i}(u,u)}
    =
    \frac{\tilde{a}_{\sigma,\tilde{\Omega}_i}(p^{(1)},p^{(1)})}{\tilde{b}_{\tilde{\Omega}_i}(p^{(1)},p^{(1)})}
    \ge
    \min_{u \in V_{h,\#}(\tilde{\Omega}_i)} \frac{\tilde{a}_{\sigma,\tilde{\Omega}_i}(u,u)}{\tilde{b}_{\tilde{\Omega}_i}(u,u)}
    =
    0
    ,
  \end{equation}
  where the last equality holds by definition of the shift \(\sigma\). However, due to the subset property, \(V_{h,\#}(\tilde{\Omega}_i) \subset V_{h}(\tilde{\Omega}_i)\), we also have
  \begin{equation}
    \min_{u \in V_{h}(\tilde{\Omega}_i)} \frac{\tilde{a}_{\sigma,\tilde{\Omega}_i}(u,u)}{\tilde{b}_{\tilde{\Omega}_i}(u,u)}
    \le
    \min_{u \in V_{h,\#}(\tilde{\Omega}_i)} \frac{\tilde{a}_{\sigma,\tilde{\Omega}_i}(u,u)}{\tilde{b}_{\tilde{\Omega}_i}(u,u)}
    =
    0
    .
  \end{equation}
   Thus, the NGEVPs and PGEVPs have the same first eigenpair \((\tilde{r}_i E_{\boldsymbol{x}}^{\#} \psi,0)\).
\end{proof}
\begin{lemma}\label{cor:positive_semidefinite}%
  Assume \labelcref{ass:symm-trian,ass:symm-pot}, then the bilinear forms \(\tilde{a}_{\sigma,\tilde{\Omega}_i} : V_h(\tilde{\Omega}_i) \times V_h(\tilde{\Omega}_i)\) are positive semidefinite for \(i \in \tilde{\mathcal{I}}^{\textnormal{i}}\) and positive definite for \(i \in \tilde{\mathcal{I}}^{\textnormal{b}}\).
\end{lemma}
\begin{proof}
  The case of \(i \in \tilde{\mathcal{I}}^{\textnormal{i}}\) follows from \cref{lem:pgevp_ngecp_same_groundstate}. For \(i \in \tilde{\mathcal{I}}^{\textnormal{b}}\), we consider the space \(V_{h,*}(\tilde{\Omega}_i) := \big\{u \in H^1(\tilde{\Omega}_i) \, \big|\, u|_\tau \in \mathcal{P}_r \, \forall \tau \in \mathcal{T}_h(\tilde{\Omega}_i)\big\}\). By the arguments of the proof of \cref{lem:pgevp_ngecp_same_groundstate}, the unique lowest eigenpair of \(\tilde{a}_{\sigma,\tilde{\Omega}_i}\) on \(V_{h,*}(\tilde{\Omega}_i)\) is \((\tilde{r}_i E_{\boldsymbol{x}}^{\#} \psi,0)\). However, since \(V_h(\tilde{\Omega}_i) \subset V_{h,*}(\tilde{\Omega}_i)\) and \(\tilde{r}_i E_{\boldsymbol{x}}^{\#} \psi \notin V_h(\tilde{\Omega}_i)\) as \(\tilde{r}_i E_{\boldsymbol{x}}^{\#} \psi > 0\) on \(\partial_{\boldsymbol{x}}\tilde{\Omega}_i\), the lowest eigenvalue of \(\tilde{a}_{\sigma,\tilde{\Omega}_i}\) on \(V_h(\tilde{\Omega}_i)\) must be greater than zero.
\end{proof}
We verify the usual ingredients to show the \(C_0\)-stability in the following. Since \(\tilde{a}_{\sigma,\tilde{\Omega}_i}\) is positive (semi)definite on \(V_h(\tilde{\Omega}_i)\) by \cref{cor:positive_semidefinite} while \(\tilde{b}_{\tilde{\Omega}_i}\) is positive semidefinite as the sum of elementwise non-negative integrals (see \cref{def:gevps}), we define the two induced seminorms \({|v|}_{\tilde{a}_{\sigma,\tilde{\Omega}_i}} := (\tilde{a}_{\sigma,\tilde{\Omega}_i}(v,v))^{1/2}\) and \({|v|}_{\tilde{b}_{\tilde{\Omega}_i}} := (\tilde{b}_{\tilde{\Omega}_i}(v,v))^{1/2}\). We then have the following result.  
\begin{lemma}[{SPSD splitting~\cite[Def.~3.7]{bastianMultilevelSpectralDomain2022}}]\label{lem:spsd_splitting}%
  Assume \labelcref{ass:symm-pot,ass:symm-trian}, then
  \begin{equation}\label{eq:spsd-splitting}
    \sum_{i=1}^N {| \tilde{r}_i v |}^2_{\tilde{a}_{\sigma,\tilde{\Omega}_i}}
    \le
    \tilde{k}_0 {\| v \|}^2_{a_{\sigma}}
    ,\quad \forall v \in V_h
    .
  \end{equation}
\end{lemma}
\begin{proof}
  Let \(v \in V_h\). We use \(\tilde{\Omega}_i := \bigcup_{i \in \mathcal{I}_i^{\#,\textnormal{n}}} \Omega_i^{\#}\) from \cref{def:periodicNeighborhood} to obtain
  \begin{equation}
    \begin{aligned}
      \sum_{i=1}^N {| \tilde{r}_i v |}^2_{\tilde{a}_{\sigma,\tilde{\Omega}_i}}
      &=
      \sum_{i=1}^N \sum_{\{j | \Omega_j^{\#} \subset \tilde{\Omega}_i\}} \tilde{a}_{\sigma,\Omega_j^{\#}}(v|_{\Omega_j^{\#}}, v|_{\Omega_j^{\#}})
      =
      \sum_{j = 1}^{L^p} \sum_{\{i \mid \Omega^{\#}_j \subset \tilde{\Omega}_i\}} \tilde{a}_{\sigma,\Omega_j^{\#}}(v|_{\Omega_j^{\#}}, v|_{\Omega_j^{\#}})
      .
    \end{aligned}
  \end{equation}
  Using that each \(\Omega_j^{\#}\) is contained in at most \(\tilde{k}_0\) periodic neighborhoods \(\tilde{\Omega}_i\) and the positive semidefiniteness of \(\tilde{a}_{\sigma,\Omega_j^{\#}}\) on \(V_h(\Omega_j^{\#})\) (see the proof of \cref{lem:pgevp_ngecp_same_groundstate}) yields
  \begin{align}
    \sum_{j = 1}^{L^p} \sum_{\{i \mid \Omega^{\#}_j \subset \tilde{\Omega}_i\}} \tilde{a}_{\sigma,\Omega_j^{\#}}(v|_{\Omega_j^{\#}}, v|_{\Omega_j^{\#}})
      \le
      \tilde{k}_0 \sum_{j = 1}^{L^p} \tilde{a}_{\sigma,\Omega_j^{\#}}(v|_{\Omega_j^{\#}}, v|_{\Omega_j^{\#}})
      =
      \tilde{k}_0 {\| v \|}^2_{a_{\sigma}}
      .
  \end{align}
\end{proof}
The following result is also essential but requires no change in the proof since it relies on the coercivity of \(a_{\sigma}\) on \(V_h\), which is given in our case.
\begin{lemma}[{strengthened triangle inequality under the square~\cite[Def.~3.6]{bastianMultilevelSpectralDomain2022}~\cite[Lem.~7.9]{doleanIntroductionDomainDecomposition2015}}]\label{lem:strengthened-trian-eq-square}%
  For any collection of \(\{v_i\}_{i=1}^N\) with \(v_i \in V_{h,0}(\Omega_i)\), it holds that  
  \begin{equation}\label{eq:strengthened-triangle-ineq}
    {\left\| \sum_{i=1}^N v_i \right\|}_{a_\sigma}^2
    \le
    N_{\textnormal{c}} \sum_{i=1}^N {\left\| v_i \right\|}_{a_\sigma}^2
    .
  \end{equation}
\end{lemma}
The remaining part is the local stability, for which we need the following result.
\begin{lemma}[trivial kernel intersection]\label{lem:trivial_kernel_union}%
  Assume \labelcref{ass:symm-trian,ass:symm-pot}, then for the NGEVPs from \cref{def:gevps}, \(\ker \tilde{a}_{\sigma,\tilde{\Omega}_i} \cap \ker \tilde{b}_{\tilde{\Omega}_i} = \{0\}\) for all \(i \in \tilde{\mathcal{I}}^{\textnormal{i}}\).
\end{lemma}
\begin{proof}
  Let \(i \in \tilde{\mathcal{I}}^{\textnormal{i}}\). By \cref{lem:pgevp_ngecp_same_groundstate}, we know that, on \(V_h(\tilde{\Omega}_i)\), \(\ker \tilde{a}_{\sigma,\tilde{\Omega}_i} = \operatorname{span}\{\tilde{r}_i E^\#_{\boldsymbol{x}}\psi\}\). In general, \(\ker \tilde{b}_{\sigma,\tilde{\Omega}_i}\) contains at least all \(v \in \ker \tilde{\chi}_i\). However, since \(\tilde{\chi}_i\) is never a constant due to the overlapping decomposition, we have \(\tilde{\chi}_i \tilde{r}_i E^\#_{\boldsymbol{x}}\psi \notin \operatorname{span}\{\tilde{r}_i E^\#_{\boldsymbol{x}}\psi\}\), which implies that \(\ker \tilde{a}_{\sigma, \tilde{\Omega}_i} \cap \ker \tilde{b}_{\tilde{\Omega}_i} = \{0\}\).  
\end{proof}
The stable splitting for the coarse space from \cref{def:periodicCoarseSpace} is achieved with the following projection.
\begin{lemma}[{local stability}]\label{lem:localStabilityOfProjection}%
  Let \(i \in \tilde{\mathcal{I}}^{\textnormal{i}}\) be given and assume \labelcref{ass:symm-pot,ass:symm-trian}, then the local projection operator \(\tilde{\Pi}_{1,i} : V_h(\tilde{\Omega}_i) \to V_h(\tilde{\Omega}_i)\) with
  \begin{equation}
  \label{eq:projectionKernelOfA}
  \begin{aligned}
    v
    \mapsto
    \tilde{\Pi}_{1,i} v
    :=
    \tfrac{\tilde{b}_{\tilde{\Omega}_i}(v,\tilde{r}_i E^{\#}_{\boldsymbol{x}}\psi)}{\tilde{b}_{\tilde{\Omega}_i}(\tilde{r}_i E^{\#}_{\boldsymbol{x}}\psi,\tilde{r}_i E^{\#}_{\boldsymbol{x}}\psi)} \tilde{r}_i E^\#_{\boldsymbol{x}}\psi
    ,
  \end{aligned}
  \end{equation}
    satisfies
  \begin{equation}
  \begin{aligned}
    {|\tilde{\Pi}_{1,i} v|}_{\tilde{a}_{\sigma,\tilde{\Omega}_i}} \le {|v|}_{\tilde{a}_{\sigma,\tilde{\Omega}_i}}
    ,\quad
    {|v - \tilde{\Pi}_{1,i} v|}_{\tilde{a}_{\sigma,\tilde{\Omega}_i}} \le {|v|}_{\tilde{a}_{\sigma,\tilde{\Omega}_i}}
    \quad \forall v \in V_h(\tilde{\Omega}_i)
    ,
    \\
    {|v - \tilde{\Pi}_{1,i} v|}_{\tilde{b}_{\tilde{\Omega}_i}}^2
    \le
    \frac{1}{\tilde{\lambda}_i^{(2)}} {|v - \tilde{\Pi}_{1,i} v|}_{\tilde{a}_{\sigma,\tilde{\Omega}_i}}^2 \quad \forall v \in V_h(\tilde{\Omega}_i)
    ,
  \end{aligned}
  \end{equation}
  with \(\tilde{\lambda}_i^{(2)} > 0\) from \cref{def:gevps}.
\end{lemma}
\begin{proof}
  The proof follows by applying~\cite[Lem.~3.15]{bastianMultilevelSpectralDomain2022}, where we have the trivial kernel intersection satisfied by \cref{lem:trivial_kernel_union}, the first unique eigenfunction given by \(\tilde{p}_i^{(1)} = \tilde{r}_i E^\#_{\boldsymbol{x}}\psi\) corresponding to a zero eigenvalue, and the positive semidefiniteness of \(\tilde{a}_{\sigma,\tilde{\Omega}_i}\) and \(\tilde{b}_{\tilde{\Omega}_i}\) on \(V_h(\tilde{\Omega}_i)\).
\end{proof}
Note that the periodic coarse space only contains components for subdomains \(\Omega_i\), whose periodic neighborhood \(\tilde{\Omega}_i\) does not touch the \(\boldsymbol{x}\)-boundary. No projection must be applied for the remaining subdomains.
The local stability of the projections \cref{eq:projectionKernelOfA} allows for a stable decomposition.
\begin{theorem}\label{thm:stable-splitting}%
  Assume \labelcref{ass:symm-pot,ass:symm-trian}, let \(v \in V_h\), then the splitting defined by
  \begin{equation}\label{eq:stableSplitting}
    v_0 = \sum_{i \in \tilde{\mathcal{I}}^{\textnormal{i}}} \chi_i \tilde{\Pi}_{1,i} \tilde{r}_i v
    \in
    V_0
    ,\quad
    v_i =
    \begin{cases}
      \chi_i (1 - \tilde{\Pi}_{1,i}) \tilde{r}_i v & i \in \tilde{\mathcal{I}}^{\textnormal{i}} \\
      \chi_i \tilde{r}_i v & i \in \tilde{\mathcal{I}}^{\textnormal{b}}
    \end{cases}
    \in
    V_{h,0}(\Omega_i)
    ,
  \end{equation}
  is \(C_0\)-stable with \(C_0 = 2 + C_1 \tilde{k}_0 (2 N_{\textnormal{c}} + 1)\), \(C_1 := (\min{\{\min_{i \in \tilde{\mathcal{I}}^{\textnormal{i}}} \tilde{\lambda}_i^{(2)}, \min_{i \in \tilde{\mathcal{I}}^{\textnormal{b}}} \tilde{\lambda}_i^{(1)}\}})^{-1}\).  
\end{theorem}
\begin{proof}
  First, the splitting is consistent since, for all \(v \in V_h\), we have
  \(v_0 + \sum_{i=1}^N v_i
  =
  \sum_{i \in \tilde{\mathcal{I}}^{\textnormal{i}}} \chi_i \tilde{r}_i v
  +
  \sum_{i \in \tilde{\mathcal{I}}^{\textnormal{b}}} \chi_i \tilde{r}_i v
  =
  \sum_{i=1}^N \chi_i \tilde{r}_i v
  =
  \sum_{i=1}^N \chi_i r_i v
  =
  v\)
  using that \(\tilde{r}_i = r_i\) on \(\operatorname{supp}(\chi_i)\). We have, using the strategy of~\cite[p13]{bastianMultilevelSpectralDomain2022}, that the splitting is locally stable since \(\forall i \in \tilde{\mathcal{I}}^{\textnormal{i}}\), the local stability (\cref{lem:localStabilityOfProjection}), and the definition of \(\tilde{b}_{\tilde{\Omega}_i}\) yields
  \begin{equation}\label{eq:local_energy_bound_inner}
  \begin{aligned}
    {\|v_i\|}_{a_{\sigma}}^2
    &=
    {|\chi_i (1 - \tilde{\Pi}_{1,i}) \tilde{r}_i v|}_{a_{\sigma,\Omega_i}}^2
    =
    {|\tilde{\chi}_i (1 - \tilde{\Pi}_{1,i}) \tilde{r}_i v|}_{\tilde{a}_{\sigma,\tilde{\Omega}_i}}^2
    \\
    &=
    {|\tilde{\chi}_i (1 - \tilde{\Pi}_{1,i}) \tilde{r}_i v|}_{\tilde{a}_{0,\tilde{\Omega}_i}}^2
    -
    \sigma {\|\tilde{\chi}_i (1 - \tilde{\Pi}_{1,i}) \tilde{r}_i v\|}_{L^2(\tilde{\Omega}_i)}^2
    \le
    {|\tilde{\chi}_i (1 - \tilde{\Pi}_{1,i}) \tilde{r}_i v|}_{\tilde{a}_{0,\tilde{\Omega}_i}}^2
    \\
    &=
    {|(1 - \tilde{\Pi}_{1,i}) \tilde{r}_i v|}_{\tilde{b}_{\tilde{\Omega}_i}}^2
    \le
    \frac{1}{\tilde{\lambda}_i^{(2)}} {|(1 - \tilde{\Pi}_{1,i}) \tilde{r}_i v|}_{\tilde{a}_{\sigma,\tilde{\Omega}_i}}^2
    \le
    \frac{1}{\tilde{\lambda}_i^{(2)}} {|\tilde{r}_i v|}_{\tilde{a}_{\sigma,\tilde{\Omega}_i}}^2
    .
  \end{aligned}
  \end{equation}
  For all boundary domains with \(i \in \tilde{\mathcal{I}}^{\textnormal{b}}\), we apply~\cite[Lem.~3.15]{bastianMultilevelSpectralDomain2022} using the projection on zero to obtain
  \begin{equation}\label{eq:local_energy_bound_outer}
  \begin{aligned}
    {\|v_i\|}_{a_{\sigma}}^2
    &=
    {|\chi_i \tilde{r}_i v|}_{a_{\sigma,\Omega_i}}^2
    =
    {|\tilde{\chi}_i \tilde{r}_i v|}_{\tilde{a}_{\sigma,\tilde{\Omega}_i}}^2
    =
    {|\tilde{\chi}_i \tilde{r}_i v|}_{\tilde{a}_{0,\tilde{\Omega}_i}}^2
    -
    \sigma {\|\tilde{\chi}_i \tilde{r}_i v\|}_{L^2(\tilde{\Omega}_i)}^2
    \\
    &\le
    {|\tilde{\chi}_i \tilde{r}_i v|}_{\tilde{a}_{0,\tilde{\Omega}_i}}^2
    =
    {| \tilde{r}_i v|}_{\tilde{b}_{\tilde{\Omega}_i}}^2
    \le
    \frac{1}{\tilde{\lambda}_i^{(1)}} {| \tilde{r}_i v|}_{\tilde{a}_{\sigma,\tilde{\Omega}_i}}^2
    .
  \end{aligned}
  \end{equation}
  Similar to~\cite[Lemma 2.9]{spillaneAbstractRobustCoarse2014}, using \cref{eq:local_energy_bound_inner,eq:local_energy_bound_outer}, the SPSD splitting \cref{eq:spsd-splitting}, and \(\tilde{\mathcal{I}}^{\textnormal{i}} \cap \tilde{\mathcal{I}}^{\textnormal{b}} = \emptyset\) yields
  \begin{equation}\label{eq:sum-of-local-energies-upper-bound}
  \begin{aligned}
    \sum_{i=1}^N {\|v_i\|}_{a_{\sigma}}^2
    =
    \sum_{i \in \tilde{\mathcal{I}}^{\textnormal{i}}} {\|v_i\|}_{a_{\sigma}}^2
    +
    \sum_{i \in \tilde{\mathcal{I}}^{\textnormal{b}}} {\|v_i\|}_{a_{\sigma}}^2
    \le
    \sum_{i \in \tilde{\mathcal{I}}^{\textnormal{i}}} \frac{1}{\tilde{\lambda}_i^{(2)}} {|\tilde{r}_i v|}_{a_{\sigma,\tilde{\Omega}_i}}^2
    +
    \sum_{i \in \tilde{\mathcal{I}}^{\textnormal{b}}} \frac{1}{\tilde{\lambda}_i^{(1)}} {|\tilde{r}_i v|}_{a_{\sigma,\tilde{\Omega}_i}}^2
    \\
    \le
    \frac{1}{{\displaystyle \min_{i \in \tilde{\mathcal{I}}^{\textnormal{i}}} \tilde{\lambda}_i^{(2)}}} \sum_{i \in \tilde{\mathcal{I}}^{\textnormal{i}}} {|\tilde{r}_i v|}_{a_{\sigma,\tilde{\Omega}_i}}^2
    +
    \frac{1}{{\displaystyle \min_{i \in \tilde{\mathcal{I}}^{\textnormal{b}}} \tilde{\lambda}_i^{(1)}}} \sum_{i \in \tilde{\mathcal{I}}^{\textnormal{b}}} {|\tilde{r}_i v|}_{a_{\sigma,\tilde{\Omega}_i}}^2
    \le
    C_1 \sum_{i=1}^N {|\tilde{r}_i v|}_{a_{\sigma,\tilde{\Omega}_i}}^2
    \le
    C_1 \tilde{k}_0 {\|v\|}_{a_{\sigma}}^2
    ,
  \end{aligned}
  \end{equation}
  with \(C_1 := (\min{\{\min_{i \in \tilde{\mathcal{I}}^{\textnormal{i}}} \tilde{\lambda}_i^{(2)}, \min_{i \in \tilde{\mathcal{I}}^{\textnormal{b}}} \tilde{\lambda}_i^{(1)}\}})^{-1}\). Then, we continue with  
  \begingroup{}%
  \allowdisplaybreaks{}%
  \begin{align}\label{eq:upper-bound-coarse-space-component}
    {\|v_0\|}_{a_{\sigma}}^2
    & =
    {\left\|v - \sum_{i=1}^N v_i\right\|}_{a_{\sigma}}^2
    \le
    2 {\|v\|}_{a_{\sigma}}^2 + 2 {\left\|\sum_{i=1}^N v_i\right\|}_{a_{\sigma}}^2
    \le
    2 {\|v\|}_{a_{\sigma}}^2 + 2 N_{\textnormal{c}} \sum_{i=1}^N {\left\|v_i\right\|}_{a_{\sigma}}^2
    \\
    & \le
    2 (1 + N_{\textnormal{c}} C_1 \tilde{k}_0) {\|v\|}_{a_{\sigma}}^2
    \notag
    ,
  \end{align}
  \endgroup{}%
  using the relation \cref{eq:sum-of-local-energies-upper-bound} and the strengthened triangle inequality (\cref{lem:strengthened-trian-eq-square}). Combining the relations \cref{eq:sum-of-local-energies-upper-bound,eq:upper-bound-coarse-space-component} yields the \(C_0\)-stability as
  \begin{align}
    \sum_{i=1}^N {\|v_i\|}_{a_{\sigma}}^2 + {\|v_0\|}_{a_{\sigma}}^2
    \le
    (2 + C_1 \tilde{k}_0 (2 N_{\textnormal{c}} + 1)) {\|v\|}_{a_{\sigma}}^2
    .
  \end{align}
\end{proof}
With the \(C_0\)-stable decomposition, we finally obtain the following.
\begin{theorem}%
  Assume \labelcref{ass:symm-pot,ass:symm-trian}, let \(V_0\) be given by \cref{def:periodicCoarseSpace}, \(\boldsymbol{M}_{\textnormal{AS},2}^{-1}\) by \cref{eq:twoLevelASMPreconditioner}, \(N_{\textnormal{c}}\) as in \cref{def:finite-coloring}, and \(C_0\) as in \cref{thm:stable-splitting}. Then, the two-level additive Schwarz method satisfies the condition number bound
  \begin{equation}\label{eq:conditionNumberBound}
    \kappa(\boldsymbol{M}_{\textnormal{AS},2}^{-1} \boldsymbol{A}_\sigma) \le  C_0^2 (N_{\textnormal{c}} + 1).
  \end{equation}
\end{theorem}
\begin{proof}
  The proof follows from~\cite[Cor.~5.12]{doleanIntroductionDomainDecomposition2015} and \cref{thm:stable-splitting}.
\end{proof}
\begin{remark}
  The condition number bound \cref{eq:conditionNumberBound} becomes \(L\)-uniform for the natural decomposition using the \(L^p\) shifted unit-cells \(\{\Omega_i^{\#}\}_{i=1}^N\) from \cref{eq:periodic-cells-definition}, since \(N_{\textnormal{c}}\), \(\tilde{k}_0\), and \(C_1\) are \(L\)-invariant in that case.  
  While our investigation is centered on demonstrating \(L\)-robustness, it is essential to note that \(C_0\) is still affected by other geometric factors and the selected partition of unity \(\{\chi_i\}_{i=1}^N\). As shown in \cref{sec:parameter_study}, for distance-based \(\chi_i\), the gradients \(\{\nabla \chi_i\}_{i=1}^N\) dependent on the mesh size \(h\) and the overlap \(\delta\), and influence the local eigenvalues and, consequently, the rate of convergence.  
  Although it seems intuitive to incorporate more than just the first eigenfunction to reduce \(C_1\), it is essential to note that in practical applications, we do not need to solve the local GEVPs since the first eigenfunctions are given a priori (see \cref{ssec:coarse-space}).  
\end{remark}

\section{Numerical experiments}\label{sec:numericalExperiments}
In this section, we evaluate the performance of the proposed PerFact preconditioner from \cref{def:algebraic-perfact-coarse-space} and provide numerical evidence for the theoretical results from \cref{sec:analysis}. We choose the \texttt{Gridap} finite element framework~\cite{badiaGridapExtensibleFinite2020} within the Julia~\cite{bezansonJuliaFreshApproach2017} language due to its high-level interface to directly specify weak forms. In all the following tests, we use the CG and GMRES implementations of the \texttt{IterativeSolvers.jl}\footnote{\url{https://github.com/JuliaLinearAlgebra/IterativeSolvers.jl}, version \texttt{v0.9.3}.} package. We mainly compare iteration numbers as a critical factor for scalability. Though actual runtimes vary with implementation, the method inherits typical DD scalability concerning computational complexity. The PerFact coarse space adds an \(\mathcal{O}(1)\) setup cost (since the unit cell is \(L\)-independent) and a slight overhead from factorizing the coarse matrix \(\boldsymbol{A}_{\sigma,0}\) (which is negligible if \(N \ll n\)). For reproducibility, the source code and all metadata are available at~\cite{theisenDdEigenLabJlDomainDecomposition2024}.  

\subsection{Source problem with the two-level preconditioner}
We first evaluate the performance of the coarse space for the solution of a shifted Schrödinger-type source problem on a two-dimensional rectangle \(\Omega_L = (0,L) \times (0,1)\) with a constant source term \(f(x,y) = 1\): Find \(u_h \in V_h \subset H^1_0(\Omega_L)\) such that
\begin{equation}\label{eq:test_source_problem}
  a_{\sigma}(u_h,v_h) = {(f, v_h)}_{L^2(\Omega_L)}, \quad \forall v_h \in V_h
  ,
\end{equation}
where \(a_{\sigma}\) is given by \cref{eq:shifted_bilinear_form}. The potential is \(V(x,y) = 10^2 (\sin^2{(\pi x)})^2 (\sin^2{(\pi y)})^2\) fulfilling \labelcref{ass:pot-dir-periodic,ass:pot-ess-bounded,ass:symm-pot,ass:symm-trian}.  
The calculations use \(\mathbb{Q}_1\) finite elements on a regular cartesian grid with mesh size \(h=1/30\). The shift \(\sigma \approx 19.32644\) and the periodic solution \(\psi_h\) are obtained on the unit cell of the same mesh with periodic boundary conditions in the \(x\)-direction, according to \cref{eq:weak-form-cell-problem}.
\par
Thus, this isolated setup of \cref{eq:test_source_problem} simulates one application of the QOSI preconditioner \(\boldsymbol{A}_{\sigma}^{-1}\) applied to the right-hand-side vector \(\boldsymbol{f}\) generated by \(f(x,y)\) and only tests the linear system solution rather than the complete eigenvalue algorithm. For each domain size \(L\), we choose a non-overlapping, structured decomposition into \(N=L\) subdomains, \(\{\Omega_i'\}_{i=1}^L\) by using \(\Omega_i' = (i-1,i) \times (0,1)\). Extending each domain by one layer of elements yields the overlapping \(\{\Omega_i\}_{i=1}^L\), thus covering \((i-1-h,i+h) \times (0,1)\) for \(i=2,\dots,L-1\) and the \(x\)-boundary domains accordingly.   
The two-level preconditioner \(\boldsymbol{M}_{\textnormal{AS},2}^{-1}\) is then used within the CG method, which uses an initial guess of one and the relative residual condition \(\texttt{rTOL} = 10^{-8}\), where \({\|\boldsymbol{r}_k\|}_2 \le \texttt{rTOL} {\|\boldsymbol{r}_0\|}_2\) to determine the convergence.

\subsubsection{Convergence rate comparison}
\cref{fig:cg_convergence_with_coarse_space} compares the resulting relative residuals within the convergence history between the one-level and two-level PerFact preconditioner. For \(\boldsymbol{M}_{\textnormal{AS},1}^{-1}\), the drastic increase in iteration numbers for \(L \to \infty\) is visible, while the coarse space within \(\boldsymbol{M}_{\textnormal{AS},2}^{-1}\) leads to a bounded convergence rate, which is independent of \(L\), confirming our theory from \cref{sec:analysis}.
\par
\begin{figure}[t]%
  \centering%
  \newcommand{\datapath}{./plots/iteration-numbers/data}%
  \input{plots/iteration-numbers/iteration-numbers.tex}%
  \caption{%
    CG residual norms for varying domain length \(L\) using the AS preconditioner with no (left) and the PerFact coarse space (right).%
  }\label{fig:cg_convergence_with_coarse_space}%
\end{figure}
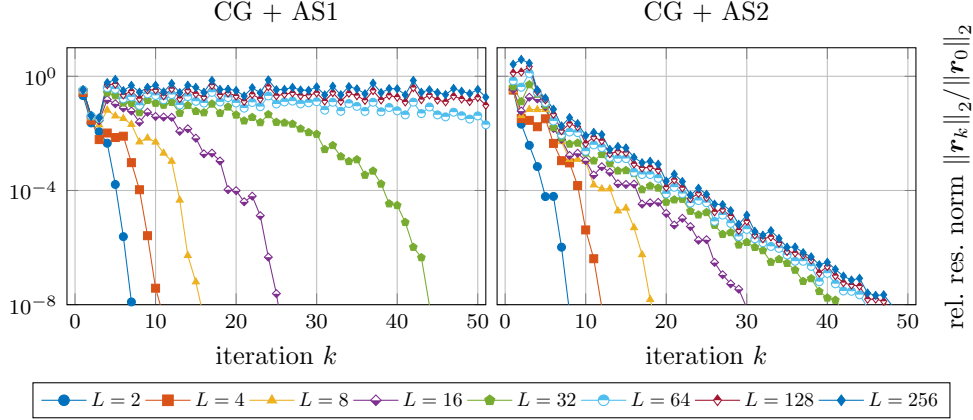%

\subsubsection{Parameter study}\label{sec:parameter_study}
We also use this test case to investigate the dependency of the methods for changing mesh sizes \(h \in \{1/10,1/20,1/30\}\) and overlap layer thickness \(\delta \in \{1,2,3\}\) in units of elements. Again, we compare the CG method iteration numbers for the one-level and the PerFact two-level DD preconditioner using the same tolerance, starting vector, and problem setup. With the \(V_h\)-interpolation operator, \(I_h\), we use the distance-based (see, e.g.,~\cite[Lem.~5.7]{doleanIntroductionDomainDecomposition2015}) partition of unity \(\chi_i(\boldsymbol{z}) = I_h \circ (d_i(\boldsymbol{z})/ \sum_{i=1}^N d_i(\boldsymbol{z}))\) with \(d_i(\boldsymbol{z}) := \operatorname{dist}(\boldsymbol{z}, \partial \Omega_i \setminus \partial \Omega) \mathbf{1}_{\Omega_i}(\boldsymbol{z})\). In our case, the partition of unity resembles linear blending between two neighboring subdomains and allows keeping the gradients \(\nabla \chi_i\) constant and thus minimal.
\par
In \cref{tab:parameter_study_rectangle}, the theoretical results are confirmed for all \(\delta\) and \(h\) since the PerFact coarse space is robust w.r.t.~\(L\). For moderate \(L < 32\), both methods have comparable performance with the one-level preconditioner, even having some iterations less. This behavior is expected since the coarse space is applied additively and not multiplicatively (deflated) to resemble the case of the analysis. With the distance-based PU, increasing the overlap thickness \(\delta\) is beneficial. Interestingly, keeping the ratio between subdomain size and overlap thickness \(H/(\delta h)\) constant, i.e., \((h,\delta) \in \{(1/10,1),(1/20,2),(1/30,3)\}\), does not change the convergence rate for a fixed \(L\). For \(L,\delta=\operatorname{const}\) and \(h \to 0\), the iterations increase due to a reduced overlap thickness \(\delta h\). Although this is a classical observation for, e.g., the Poisson problem, it is not directly clear for the quasi-optimally shifted Schrödinger operator since the shift \(\sigma\) also varies when \(h\) varies \textendash\ which, in other words, means that we compare different operators in that case.  
\par
\begin{table}[t]%
  \footnotesize%
  \centering%
  \caption{%
  CG iteration numbers for relative residuals to converge to \(\texttt{rTOL} = 10^{-8}\) using no coarse space (first number) and the PerFact coarse space (second number).%
  }\label{tab:parameter_study_rectangle}%
  \resizebox{\columnwidth}{!}{%
  \pgfplotstabletypeset[
    col sep = comma,
    every head row/.style={
      output empty row,
      before row={
        \toprule
        & \multicolumn{3}{c}{\(h = 1/10\)} & \multicolumn{3}{c}{\(h = 1/20\)} & \multicolumn{3}{c}{\(h = 1/30\)} \\
        \cmidrule(lr){2-4} \cmidrule(lr){5-7} \cmidrule(lr){8-10}
        \(L\) & \(\delta = 1\) & \(\delta = 2\) & \(\delta = 3\) & \(\delta = 1\) & \(\delta = 2\) & \(\delta = 3\) & \(\delta = 1\) & \(\delta = 2\) & \(\delta = 3\) \\
      },
      after row=\midrule
    },
    every last row/.style={after row=\bottomrule},
    every column/.style={string type,column type = {l}},
    columns/Lx/.style={string type,column type = {r}},
  ]%
  {./tables/data-iteration-numbers.csv}%
  }
\end{table}

\subsection{Linear chain model with Coulomb potential}
The initial hyperbox setup of \(\Omega_L\) from \cref{eq:geometry-description} can be generalized to more complex domains \(\hat{\Omega}\), which are generated by translated copies of a unit cell. Then, there exists an enclosing box such that \(\hat{\Omega} \subset \Omega_L\) for some \(L\). Thus, it is possible to consider a hypothetical penalty potential \(\hat{V}(\boldsymbol{z}; V, a) := V(\boldsymbol{z}) + a \mathbf{1}_{\hat{\Omega}^{\textnormal{c}}}(\boldsymbol{z})\) and take \(a \to \infty\). Since the penalty is only applied in the complement \(\hat{\Omega}^{\textnormal{c}} := \Omega_L \setminus \hat{\Omega}\) when the indicator function \(\mathbf{1}_{\hat{\Omega}^{\textnormal{c}}}\) is nonzero, the resulting eigenfunctions of \cref{eq:schroedingerEquation} converge to the eigenfunctions of the problem only solved in \(\hat{\Omega}\) with zero Dirichlet conditions on \(\partial \hat{\Omega}\) (see the \emph{barrier principle} in~\cite{stammQuasiOptimalFactorizationPreconditioner2022}).

\subsubsection{Model description}
With that, we consider a linear chain of \(N\) particles, a toy model inspired by the cumulene configuration of the ideal carbyne~\cite{zhaoCarbonNanowireMade2003}. Let \(\{\boldsymbol{c}_i\}_{i=1}^N\) be a set of center positions with \(\boldsymbol{c}_i := (R+2(i-1)r,0)^T\) for some \(R > r > 0\), to construct a union of disks domain \(\hat{\Omega}_N := \bigcup_{i=1}^N B_R(\boldsymbol{c}_i)\) with \(B_R(\boldsymbol{c})\) being a disk of radius \(R\) and center \(\boldsymbol{c} \in \mathbb{R}^2\). The domain can naturally be split into DD-suitable overlapping disks, following the spirit of~\cite{cancesDomainDecompositionImplicit2013}. Each particle generates a radial potential, modeled by a truncated Coulomb potential \(V_{\textnormal{C}}\), around its center fulfilling \labelcref{ass:pot-ess-bounded}. We see that the unit cell here is \(\hat{\Omega}_0 = (R-r,R+r) \times (-R,R)\), which we complement with defect regions in the mesh at both \(x\)-ends to form a union of disks. However, the factorization theory (see the \emph{principle of defect invariance} in~\cite{stammQuasiOptimalFactorizationPreconditioner2022}) and the DD analysis from \cref{sec:analysis} still apply (assuming asymptotic insignificance of the defect regions, which is valid and could theoretically be achieved by enriching the coarse space with all DOFs in these regions). The neighboring interactions between two particles are neglected for simplicity, but we must add the two boundary ghost centers to the collection \(\boldsymbol{\mathcal{C}} = \{\boldsymbol{c}_i\}_{i=1}^N \cup \{(R-2r,0)^T, (R+2Nr,0)^T\}\) to fulfill the assumption \labelcref{ass:pot-dir-periodic}. The resulting potential then reads \(V(\boldsymbol{z}) = V_{\textnormal{C}}\left(\min_{\boldsymbol{c} \in \boldsymbol{\mathcal{C}}}{\|\boldsymbol{z} - \boldsymbol{c}\|}_2\right)\) with \(V_{\textnormal{C}}(r) = - Z / \max\{r,b\}\) for some parameters \(Z,b > 0\). A visualization of the potential and an exemplary \(\mathbb{P}_1\) finite element mesh is given in \cref{sfig:union_pot_mesh}. 

\subsubsection{Flexibility test of the coarse space within eigensolvers}
With the union of disks setup, we now test the coarse space within the SI-LOPCG solver from \cref{alg:si-lopcg}. To demonstrate the flexibility of the PerFact coarse space, we not only test the ASM2 method but also consider the stationary RAS2 iteration (multiplicative coarse correction) from \cref{eq:two-level-stationary-ras-mult-cc} as well as the RAS2 preconditioner from \cref{eq:dd-preconditioners,eq:two-level-preconditioners} with the PerFact coarse space from \cref{eq:algebraic-coarse-space} to use within the GMRES method (restart not reached). For the parameters \(Z = 1, b = 10^{-4}\), a series of computations is then performed for \(N \in \{1, 2, 4, \dots, 128\}\) on meshes with \(x\)-length proportional to \(N\) while the coarse space dimension is \(|V_0| = N + 2\). The discretization results in \(1353\) nodes for the unit cell \(\hat{\Omega}_0\), and the resulting solution is focused in the center of the domain. A visualization of the first eigenfunction for \(N = 4\) is given in \cref{sfig:first_eigenfunction}.
\par
\begin{figure}[t]
  \centering%
  \includegraphics[width=0.49\linewidth]{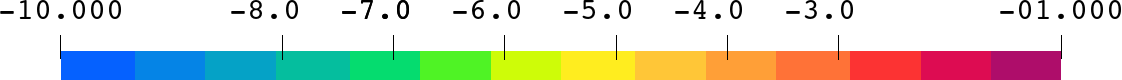}%
  \hfill%
  \includegraphics[width=0.49\linewidth]{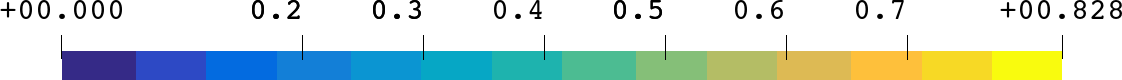}%
  \\%
  \subfloat[%
    Potential function and union of disks mesh.\label{sfig:union_pot_mesh}
  ]{\includegraphics[width=0.5\linewidth]{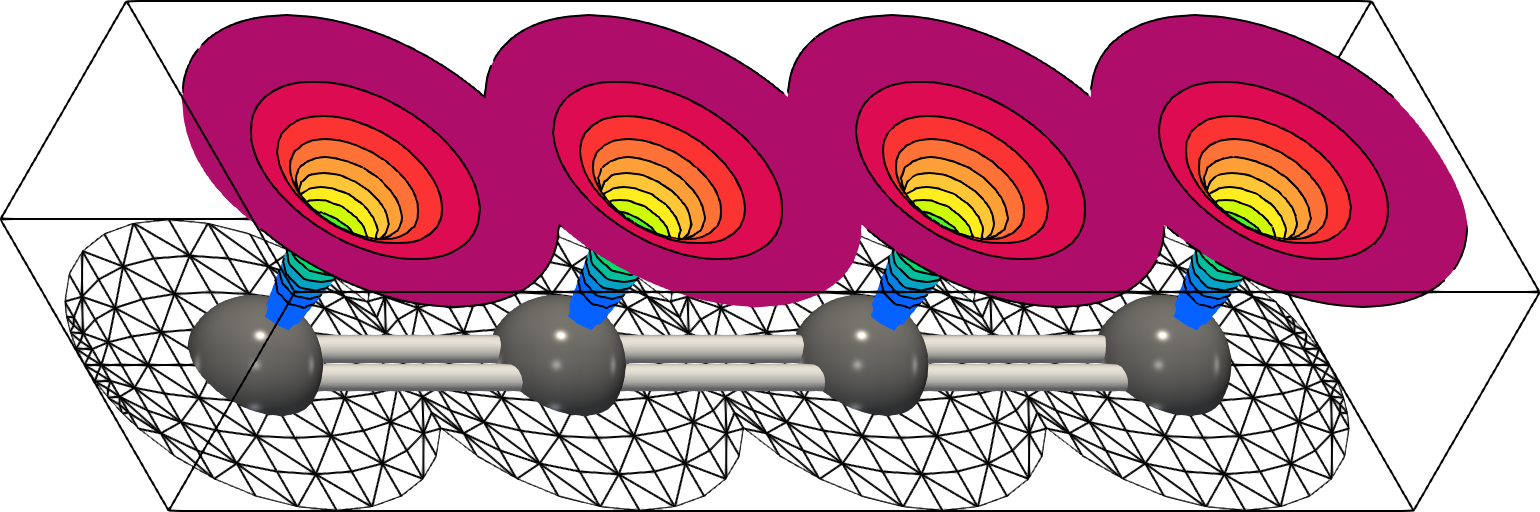}}%
  \subfloat[%
    First eigenfunction \(\phi\) for \(N=4\).\label{sfig:first_eigenfunction}%
  ]{\includegraphics[width=0.5\linewidth]{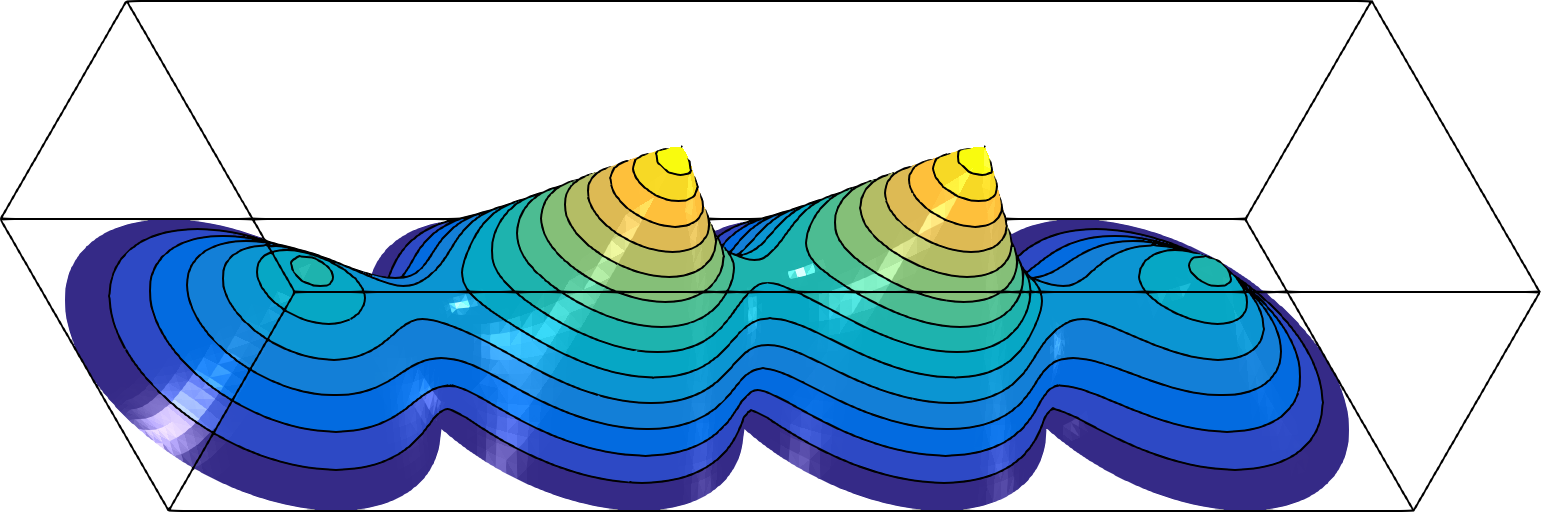}}%
  \caption{%
    A union of disks domain \(\Omega_N\) for \(N=4\) with \textnormal{\textbf{(a)}} the applied symmetric potential \(V\) (truncated to \([-10,-1]\) for visualization) and an exemplary \(\mathbb{P}_1\) mesh and \textnormal{\textbf{(b)}} the resulting first eigenfunction \(\phi\). Both color scales divided the listed interval into 14 colors.%
  }\label{fig:union-of-disks-sketch}%
\end{figure}
\par
The resulting iteration numbers are reported in \cref{tab:silopcg_iteration_numbers}. First, we observe that the one-level RAS1 is, as expected, not robust for increasing \(L\), such that the computations are skipped for \(L > 32\). The stationary RAS2 method performs reasonably well and can be used as an inner solver. However, the Krylov accelerated CG+ASM2 and GMRES+RAS2 variants work better, having much fewer inner iterations per outer step. They are also robust with respect to \(L\) since the maximum number of inner iterations, denoted by \(\max_{\textnormal{i}}\), is bounded from above. Combined with bounded outer iterations obtained by the QOSI strategy, this yields a bounded number of inner iterations, denoted by \(\sum_{\textnormal{i}}\), which measures computational cost. We also observe that the unsymmetric GMRES+RAS2 case has fewer iterations than the symmetric CG+ASM2 case \textendash\ which is the expected behavior. Furthermore, the outer iterations are not influenced by the inner solver. Another interesting observation relates to the case when the maximum number of inner iterations \(k_{\max} = 10000\) is reached for the RAS1 (\(N = 32\)) method. Here, we observe that even in that case, convergence of the eigenvalue solver can still be achieved. This opens questions about the interplay between inner and outer tolerances, which we discuss in the following.
\par
\begin{table}[t]
  \footnotesize%
  \centering%
  \caption{%
  Inner and outer iteration numbers of the SI-LOPCG method (\(\boldsymbol{x}_0 = \boldsymbol{e}_1\), \(\boldsymbol{x}_1 = \boldsymbol{1}/\|\boldsymbol{1}\|_{\boldsymbol{B}}\)) using the stationary RAS1, stationary RAS2, CG+ASM2, and GMRES+RAS2 as inner solvers (with zero initial guesses). Skipped simulations are indicated with \(\dagger\). We apply a relative tolerance of \(\texttt{rTOL}_{\textnormal{i}} = 10^{-8}\) for the inner residuals and an absolute tolerance (since the eigenvector is normalized after each iteration) of \(\texttt{TOL}_{\textnormal{o}} = 10^{-8}\) for the outer spectral residuals. For the inner solver, \(k_{\max} = 10000\) applies, and \(10000^*\) is displayed when \(k_{\max}\) is reached. We abbreviate with \(\textnormal{it}_{\textnormal{o}}\) the outer iterations, with \(\max_{\textnormal{i}}\) the maximal number of inner iterations as a measure for the worst case, and with \(\sum_{\textnormal{i}}\) the sum of all inner iterations (approximately computational costs).
  We note that the slight increase in the number of iterations for higher \(L\) values is not due to the inner convergence rate but rather due to increased norms of the initial values (like in \cref{fig:cg_convergence_with_coarse_space}), which, in turn, naturally stem from changes in the computational domain. %
  }\label{tab:silopcg_iteration_numbers}%
  \resizebox{\columnwidth}{!}{%
  \pgfplotstabletypeset[
    col sep = comma,
    every head row/.style={
      output empty row,
      before row={
        \toprule
        & \multicolumn{3}{c}{\texttt{RAS1}} & \multicolumn{3}{c}{\texttt{RAS2}} & \multicolumn{3}{c}{\texttt{CG+ASM2}} & \multicolumn{3}{c}{\texttt{GMRES+RAS2}} \\
        \cmidrule(lr){2-4} \cmidrule(lr){5-7} \cmidrule(lr){8-10} \cmidrule(lr){11-13}
        \(N \sim L\) & \(\textnormal{it}_{\textnormal{o}}\) & \(\max_{\textnormal{i}}\) & \(\sum_{\textnormal{i}}\) & \(\textnormal{it}_{\textnormal{o}}\) & \(\max_{\textnormal{i}}\) & \(\sum_{\textnormal{i}}\) & \(\textnormal{it}_{\textnormal{o}}\) & \(\max_{\textnormal{i}}\) & \(\sum_{\textnormal{i}}\) & \(\textnormal{it}_{\textnormal{o}}\) & \(\max_{\textnormal{i}}\) & \(\sum_{\textnormal{i}}\) \\
      },
      after row=\midrule
    },
    every last row/.style={after row=\bottomrule},
    every column/.style={string type,column type = {l}},
    columns/Lx/.style={string type,column type = {r}},
    columns/ras1_it_outer/.style={string replace={0}{\(\dagger\)}},
    columns/ras1_maxit_inner/.style={string replace={0}{\(\dagger\)},string replace={10000}{10000*}},
    columns/ras1_sumit_inner/.style={string replace={0}{\(\dagger\)}},
    columns/ras2_maxit_inner/.style={string replace={10000}{10000*}},
  ]%
  {./tables/data-union.csv}%
  }%
\end{table}

\subsection{Fusing the loops}
Instead of the fixed tolerance of \(\texttt{rTOL}_{\textnormal{i}} = 10^{-8}\), we can also choose it proportionally to the outer residual norm as \(\texttt{rTOL}_{\textnormal{i}} = \min\{0.1,{\|\boldsymbol{r}_{\textnormal{o},k}\|}_2\}\), inspired by~\cite{freitagConvergenceInexactInverse2007}. We compare these two cases (fixed/adaptive) with a third approach, which skips the inner Krylov solver and directly uses \(\boldsymbol{M}_{\textnormal{RAS},2}^{-1}\) as a preconditioner in the LOPCG method, i.e., solving \(\boldsymbol{M}_{\textnormal{RAS},2} \boldsymbol{w}_k = \boldsymbol{r}_{\textnormal{o},k}\) in line~\ref{algstep:inner-loop} of \cref{alg:si-lopcg} (instead of \(\boldsymbol{A}_{\sigma} \boldsymbol{w}_k = \boldsymbol{r}_{\textnormal{o},k}\)). This strategy still has the shift-and-invert effect and does not suffer from a high condition number.
\par
To render the case more complex, a three-dimensional plane-like box \(\Omega_L = {(0,L)}^2 \times (0,1)\) is considered with a potential \(V\) that is periodic in the \(x\)- and \(y\)-directions and has a linear gradient in the \(z\)-direction. On purpose, we now choose an unsymmetric potential (not fulfilling \labelcref{ass:symm-pot}), which is given by
\begin{equation}
  V(x,y,z) = 10 (4 + \sin{2 \pi x} + \sin{4 \pi x} + 2 \sin{2 \pi y} + 2 \sin{4 \pi y} + z)
  .
\end{equation}
For the domain decomposition, we also apply the most general case by using an unstructured METIS partition. \cref{fig:fusing_sketch} presents the partition, \(V\), and an exemplary first eigenfunction for the case of \(L=2\). With a structured \(\mathbb{Q}_1\)-discretization using \(h=1/10, \delta=1,\texttt{TOL}_{\textnormal{o}}=10^{-10}\) and inner initial guess of zero, we perform a series of computations for \(L \in \{4,8,16,32\}\) and keep track of all inner iterations. As we observe in \cref{fig:fusing_iteration_numbers}, the method is still robust w.r.t.~to \(L\), even for the unsymmetric potential case. Although the adaptive strategy can significantly reduce the number of iterations, the fused approach leads to the fastest convergence. The idea of fusing the inner and the outer loop could even be carried further to the case of nonlinear eigenvalue problems, which have an additional third loop to handle the nonlinearity.  
\par
\begin{figure}[t]
  \centering%
  \subfloat[%
    METIS partition of \(\Omega_2\).\label{sfig:fusing_partition}
  ]{\includegraphics[width=0.3333\linewidth]{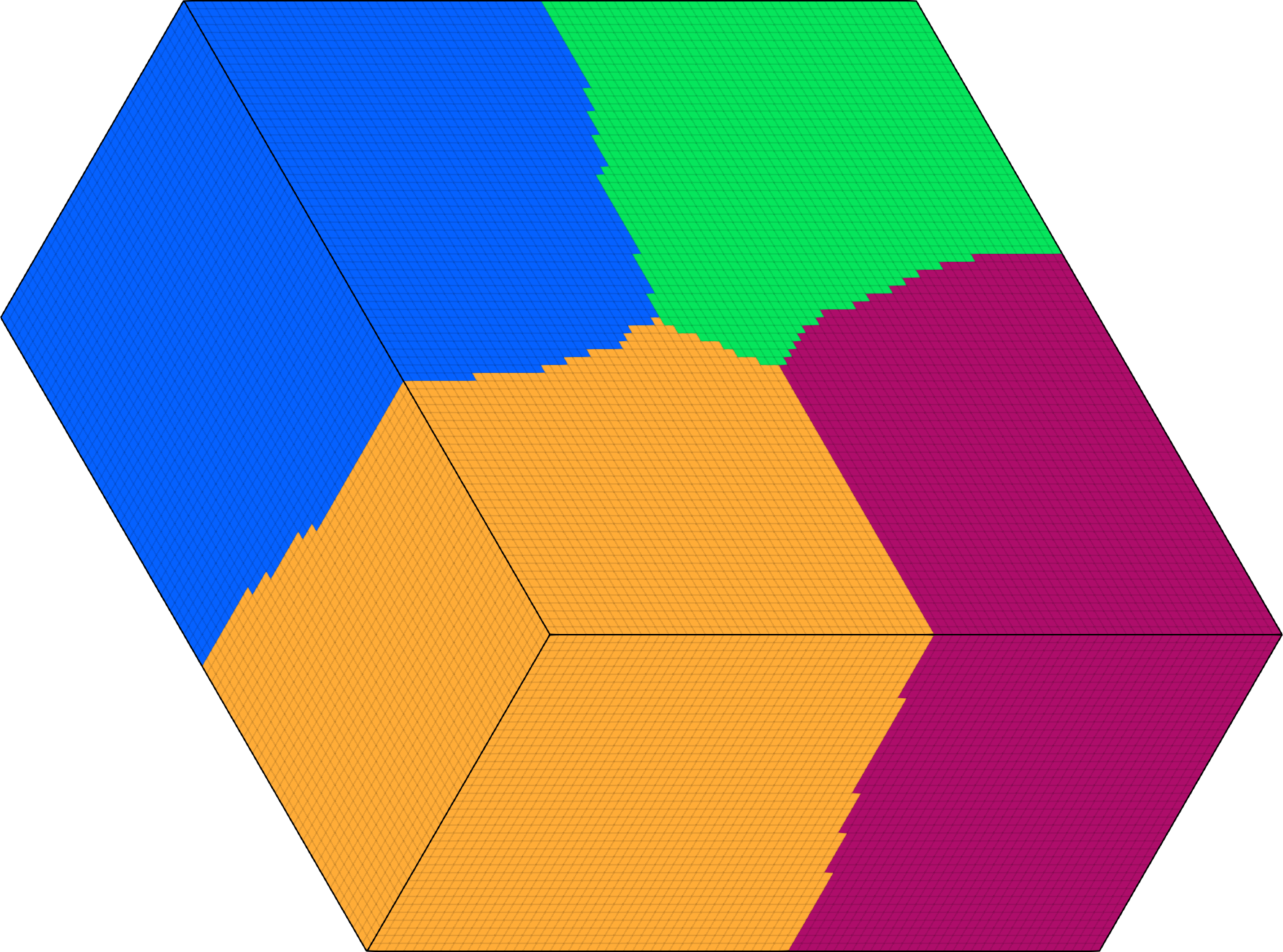}}%
  \hfill%
  \subfloat[%
    Potential \(V\) (unsymmetric).\label{sfig:fusing_potential}%
  ]{\includegraphics[width=0.3333\linewidth]{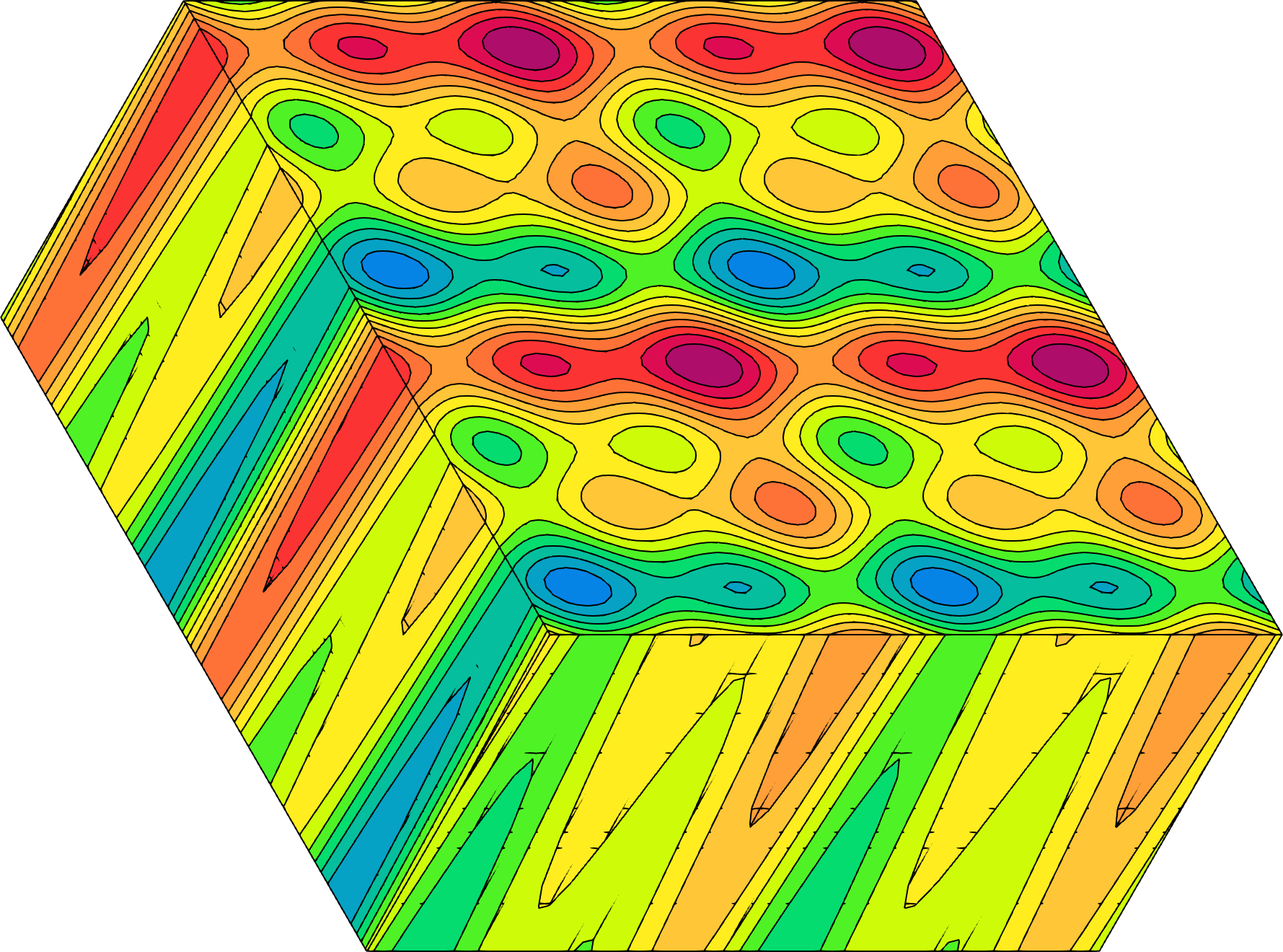}}%
  \hfill%
  \subfloat[%
    First eigenfunction (\(L=2\)).\label{sfig:fusing_first_eigenfunction}%
  ]{\includegraphics[width=0.3333\linewidth]{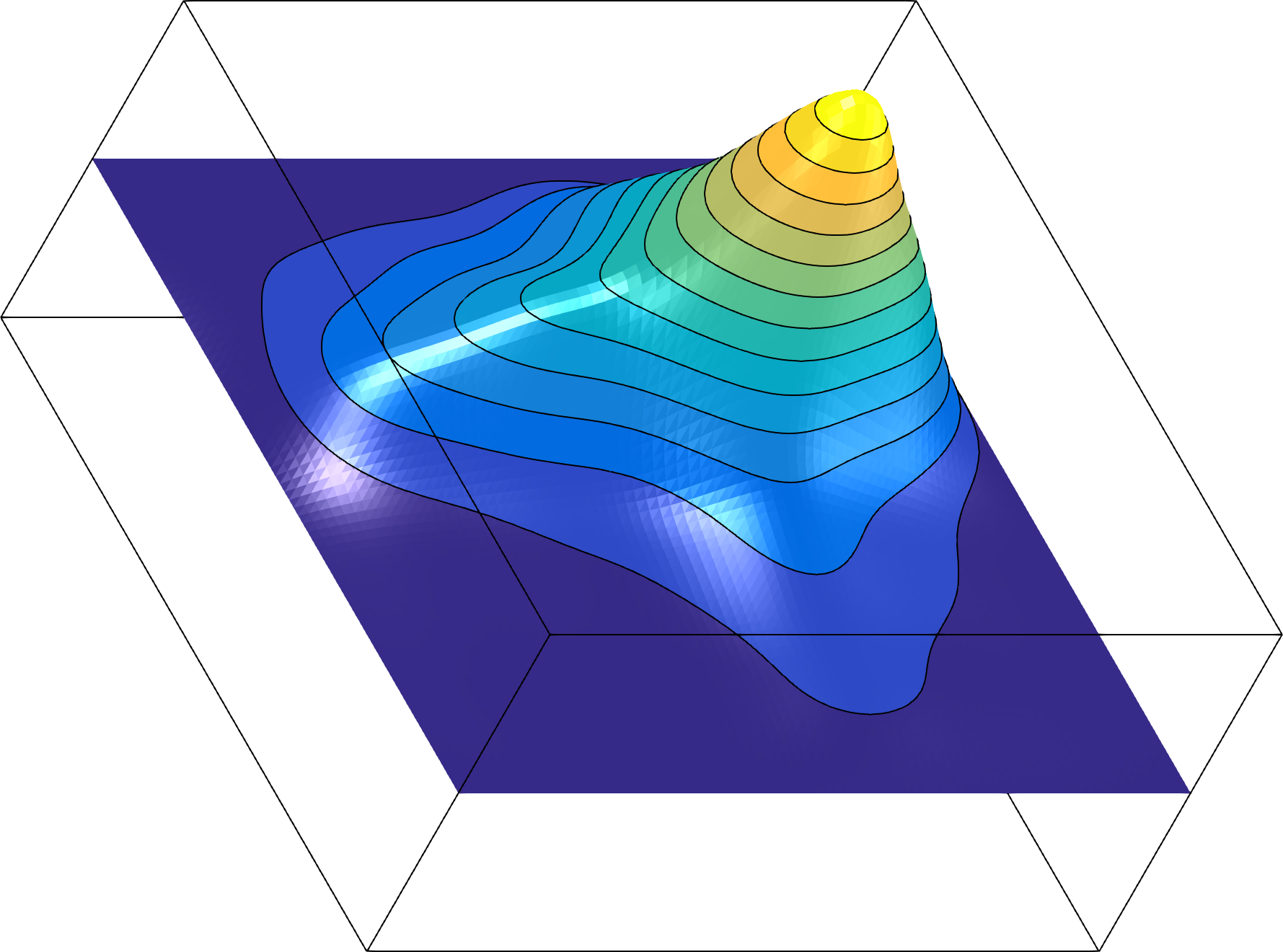}}%
  \caption{%
    \textnormal{\textbf{(a)}} An unstructured METIS element partition of \(\Omega_2\) into \(\{\Omega_i'\}_{i=1}^4\), \textnormal{\textbf{(b)}} a periodic but unsymmetric potential \(V\), and \textnormal{\textbf{(c)}} the first eigenfunction for the case of \(L=2\) with \(h=1/40\).
  }\label{fig:fusing_sketch}%
\end{figure}
\par
\par
\begin{figure}[t]%
  \centering%
  \newcommand{\datapath}{./plots/fused-iteration-numbers/data}%
  \input{plots/fused-iteration-numbers/it-nums.tex}%
  \caption{%
    Comparison of the total number of inner iterations for a fixed inner tolerance, an adaptive inner tolerance, and direct usage of the RAS2-preconditioner within the LOPCG method.%
  }\label{fig:fusing_iteration_numbers}%
\end{figure}
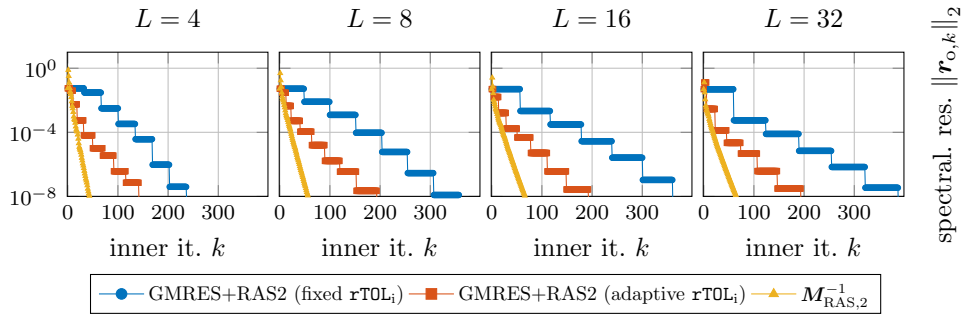%

\section{Conclusion and future work}
In this paper, we presented a new domain decomposition preconditioner for the quasi-optimally shifted Schrödinger operator, which is robust with respect to anisotropic domain expansion. This setup is motivated by 1d structures (e.g., carbon nanotubes) or 2d materials (e.g., graphene) in material science. Initially, we academically motivated the fundamental dilemma of effective shifting for eigenvalue problems in contrast to the fast convergence of iterative eigenvalue solvers. With the new factorization preconditioner, we successfully combined both aspects to simultaneously achieve fast convergence of both the eigenvalue and linear solvers. For our analysis, we utilized tools from the theory of spectral coarse spaces to prove a condition number bound for the preconditioned system. Operating in a setup where coercivity asymptotically vanishes, we took special care in adapting the geometrical setup and the theory accordingly, for instance, by considering a hypothetical alignment of the domain decomposition. Numerical experiments confirmed the theoretical results and demonstrated the performance of the coarse space within eigensolvers even when theoretical assumptions do not apply.  
\par
Limitations of the method include the assumptions on the potentials required to apply the theory and the periodicity assumption. Consequently, future work can focus on extending the theory to more general potentials (since the method also works in these cases) and applying it to more realistic nonlinear eigenvalue problems, such as tight-binding models. Furthermore, rather than solving only one problem on a unit cell to obtain a shift and coarse space, we could consider general parameter-dependent eigenvalue problems. In such a scenario, known coarse spaces from the solution of linear systems might serve as an efficient basis to extract a suitable shift for use within the eigenvalue solver. In other words, future efforts could strengthen the connection between solver theory from iterative linear systems to eigenvalue solvers.

%% file: plots/sketches/geometrical-setup.tex



\tdplotsetmaincoords{60}{135}%
\begin{tikzpicture}[tdplot_main_coords,line join=bevel, scale=0.6]%

\definecolor{color1}{rgb}{0,    ,0.4470,0.7410}
\definecolor{color2}{rgb}{0.8500,0.3250,0.0980}
\definecolor{color3}{rgb}{0.9290,0.6940,0.1250}
\definecolor{color4}{rgb}{0.4940,0.1840,0.5560}
\definecolor{color5}{rgb}{0.4660,0.6740,0.1880}
\definecolor{color6}{rgb}{0.3010,0.7450,0.9330}
\definecolor{color7}{rgb}{0.6350,0.0780,0.1840}

\pgfmathsetmacro{\N}{5}%
\pgfmathsetmacro{\elll}{2}%
\pgfmathsetmacro{\eps}{0.5} 

\filldraw[
  fill stretch image=\datapath/s-top-x1x2.png
]%
(-\eps+\eps-\eps,-\eps+\eps-\eps,\elll) -- (-\eps+\eps+\eps-\eps,-\eps+\eps-\eps,\elll) -- (-\eps+\eps+\eps-\eps,-\eps+\eps+\eps-\eps,\elll) -- (-\eps+\eps-\eps,-\eps+\eps+\eps-\eps,\elll) -- cycle;

\foreach \i in {1,...,\N}{
  \pgfmathsetmacro\iisint{abs(Mod(\i,1))<0.001}
  \pgfmathsetmacro\idist{\iisint*1+Mod(\i,1)}
  \pgfmathsetmacro\jisint{abs(Mod(\eps,1))<0.001}
  \pgfmathsetmacro\jdist{\jisint*1+Mod(\eps,1)}
  \pgfmathsetmacro\j{\eps}
  \filldraw[
    fill stretch image=\datapath/s-top-x2.png
  ]
  (0+\i-\idist,-\eps+\j-\jdist,\elll) -- (0+\idist+\i-\idist,-\eps+\j-\jdist,\elll) -- (0+\idist+\i-\idist,-\eps+\jdist+\j-\jdist,\elll) -- (0+\i-\idist,-\eps+\jdist+\j-\jdist,\elll) -- cycle;
  \filldraw[
    fill stretch image=\datapath/s-side-x2.png
  ]
  (0+\idist+\i-\idist,-\eps+\j-\jdist,0) -- (0+\idist+\i-\idist,-\eps+\j-\jdist,\elll) -- (0+\idist+\i-\idist,-\eps+\jdist+\j-\jdist,\elll) -- (0+\idist+\i-\idist,-\eps+\jdist+\j-\jdist,0) -- cycle;
}

\foreach \j in {1,2,...,\N}{
  \pgfmathsetmacro\iisint{abs(Mod(\eps,1))<0.001}
  \pgfmathsetmacro\idist{\iisint*1+Mod(\eps,1)}
  \pgfmathsetmacro\jisint{abs(Mod(\j,1))<0.001}
  \pgfmathsetmacro\jdist{\jisint*1+Mod(\j,1)}
  \pgfmathsetmacro\i{\eps}
  \filldraw[
    fill stretch image=\datapath/s-top-x1.png
  ]
  (-\eps+\i-\idist,0+\j-\jdist,\elll) -- (-\eps+\idist+\i-\idist,0+\j-\jdist,\elll) -- (-\eps+\idist+\i-\idist,0+\jdist+\j-\jdist,\elll) -- (-\eps+\i-\idist,0+\jdist+\j-\jdist,\elll) -- cycle;
  \filldraw[
    fill stretch image=\datapath/s-side-x1.png
  ]
  (-\eps+\i-\idist,0+\jdist+\j-\jdist,0) -- (-\eps+\idist+\i-\idist,0+\jdist+\j-\jdist,0) -- (-\eps+\idist+\i-\idist,0+\jdist+\j-\jdist,\elll) -- (-\eps+\i-\idist,0+\jdist+\j-\jdist,\elll) -- cycle;
}

\foreach \i in {0.5,1,2,...,\N}{
  \foreach \j in {1,2,...,\N}{
    \pgfmathsetmacro\iisint{abs(Mod(\i,1))<0.001}
    \pgfmathsetmacro\idist{\iisint*1+Mod(\i,1)}
    \pgfmathsetmacro\jisint{abs(Mod(\j,1))<0.001}
    \pgfmathsetmacro\jdist{\jisint*1+Mod(\j,1)}
    \filldraw[
      fill stretch image=\datapath/s-pot.png,
    ]
    (0+\i-\idist,0+\j-\jdist,\elll) -- (0+\idist+\i-\idist,0+\j-\jdist,\elll) -- (0+\idist+\i-\idist,0+\jdist+\j-\jdist,\elll) -- (0+\i-\idist,0+\jdist+\j-\jdist,\elll) -- cycle;
    \filldraw[
      fill stretch image=\datapath/s-pot.png
    ]
    (0+\i-\idist,0+\jdist+\j-\jdist,0) -- (0+\idist+\i-\idist,0+\jdist+\j-\jdist,0) -- (0+\idist+\i-\idist,0+\jdist+\j-\jdist,\elll) -- (0+\i-\idist,0+\jdist+\j-\jdist,\elll) -- cycle;
    \filldraw[
      fill stretch image=\datapath/s-pot.png
    ]
    (0+\idist+\i-\idist,0+\j-\jdist,0) -- (0+\idist+\i-\idist,0+\j-\jdist,\elll) -- (0+\idist+\i-\idist,0+\jdist+\j-\jdist,\elll) -- (0+\idist+\i-\idist,0+\jdist+\j-\jdist,0) -- cycle;
  }
}

\filldraw[
  fill=color1!25,
  fill stretch image=\datapath/s-pot.png,
  fill image opacity=0.25
]
(0+\N-1,0+\N-1,\elll) -- (1+\N-1,0+\N-1,\elll) -- (1+\N-1,1+\N-1,\elll) -- (0+\N-1,1+\N-1,\elll) -- cycle;
\filldraw[
  fill=color1!25,
  fill stretch image=\datapath/s-pot.png,
  fill image opacity=0.25
]
(0+\N-1,1+\N-1,0) -- (1+\N-1,1+\N-1,0) -- (1+\N-1,1+\N-1,\elll) -- (0+\N-1,1+\N-1,\elll) -- cycle;
\filldraw[
  fill=color1!25,
  fill stretch image=\datapath/s-pot.png,
  fill image opacity=0.25
]
(1+\N-1,0+\N-1,0) -- (1+\N-1,0+\N-1,\elll) -- (1+\N-1,1+\N-1,\elll) -- (1+\N-1,1+\N-1,0) -- cycle;

\draw[
  color2, thick,
] (-\eps,-\eps,\elll) -- (\N,-\eps,\elll) -- (\N,\N,\elll) -- (-\eps,\N,\elll) -- cycle;
\draw[
  color2, thick,
] (\N,-\eps,0) -- (\N,-\eps,\elll) -- (\N,\N,\elll) -- (\N,\N,0) -- cycle;
\draw[
  color2, thick,
] (\N,\N,0) -- (\N,\N,\elll) -- (-\eps,\N,\elll) -- (-\eps,\N,0) -- cycle;

\draw[
  color1
]
(0+\N-1,0+\N-1,\elll) -- (1+\N-1,0+\N-1,\elll) -- (1+\N-1,1+\N-1,\elll) -- (0+\N-1,1+\N-1,\elll) -- cycle;
\draw[
  color1
]
(0+\N-1,1+\N-1,0) -- (1+\N-1,1+\N-1,0) -- (1+\N-1,1+\N-1,\elll) -- (0+\N-1,1+\N-1,\elll) -- cycle;
\draw[
  color1
]
(1+\N-1,0+\N-1,0) -- (1+\N-1,0+\N-1,\elll) -- (1+\N-1,1+\N-1,\elll) -- (1+\N-1,1+\N-1,0) -- cycle;

\node[color1] at (\N,\N-0.5,\elll/2) {$\Omega_1$};
\node[color2] at (\N,1,-0.5) {$\Omega_L$};

\draw[
  >=latex,
  <->
] (-\eps,\N+0.5,0) -- ++(0,0,\elll) node[anchor=west, midway]{$\ell$};
\draw[] (-\eps,\N+0.5-0.25,\elll) -- ++(0,0.5,0);
\draw[] (-\eps,\N+0.5-0.25,0) -- ++(0,0.5,0);
\draw[
  >=latex,
  <->
] (-\eps,-0.5-\eps,\elll) -- ++(\N+\eps,0,0) node[anchor=south east, midway]{$L$};
\draw[] (-\eps,-0.5-\eps-0.25,\elll) -- ++(0,0.5,0);
\draw[] (\N,-0.5-\eps-0.25,\elll) -- ++(0,0.5,0);
\draw[
  >=latex,
  <->
] (-0.5-\eps,-\eps,\elll) -- ++(0,\N+\eps,0) node[anchor=south, midway]{$L$};
\draw[] (-0.5-\eps-0.25,-\eps,\elll) -- ++(0.5,0,0);
\draw[] (-0.5-\eps-0.25,\N,\elll) -- ++(0.5,0,0);

\draw[->]
  (\N,\N,\elll) -- ++ (-\N-1.0-\eps,0,0) node[anchor=south]{$x_1$};
\draw[->]
  (\N,\N,\elll) -- ++(0,-\N-1.0-\eps,0) node[anchor=north]{$x_2$};
\draw[->]
  (\N,\N,\elll) -- ++(0,0,-\elll-1.0) node[anchor=west]{$y_1$};

\draw[
  thick,
  >=latex,
  <-,
] (\N-2.5,\N,\elll/2) to ++(-0.5,1,-\elll/2);
\draw[
  thick,
  >=latex,
  <-,
] (\N-3.5,\N,\elll/2) to ++(0.5,1,-\elll/2);
\node[
  below,
  text width=2.0cm,
  align=center
] at (\N-3.0,\N+1,0)  {$V(\boldsymbol{x},\boldsymbol{y}) = V(\boldsymbol{x}+\boldsymbol{i},\boldsymbol{y})$};

\end{tikzpicture}


%% file: plots/iteration-numbers-laplace-inner-outer/iteration-numbers-laplace-inner-outer.tex



\begin{tikzpicture}

  \definecolor{color1}{rgb}{0     ,0.4470,0.7410}
  \definecolor{color2}{rgb}{0.8500,0.3250,0.0980}
  \definecolor{color3}{rgb}{0.9290,0.6940,0.1250}
  \definecolor{color4}{rgb}{0.4940,0.1840,0.5560}
  \definecolor{color5}{rgb}{0.4660,0.6740,0.1880}
  \definecolor{color6}{rgb}{0.3010,0.7450,0.9330}
  \definecolor{color7}{rgb}{0.6350,0.0780,0.1840}
  \pgfplotscreateplotcyclelist{matlab}{
    color1,every mark/.append style={solid},mark=*\\
    color2,every mark/.append style={solid},mark=square*\\
    color3,every mark/.append style={solid},mark=triangle*\\
    color4,every mark/.append style={solid},mark=halfsquare*\\
    color5,every mark/.append style={solid},mark=pentagon*\\
    color6,every mark/.append style={solid},mark=halfcircle*\\
    color7,every mark/.append style={solid,rotate=180},mark=halfdiamond*\\
    color1,every mark/.append style={solid},mark=diamond*\\
    color2,every mark/.append style={solid},mark=halfsquare right*\\
    color3,every mark/.append style={solid},mark=halfsquare left*\\
  }

  \readdef{\datapath/limit.csv}{\limit}

  \begin{axis}[
      name=myaxis,
      cycle list name=matlab,
      xlabel={shift $\sigma$},
      xmin=0,
      grid=major,
      mark size = 1.5pt,
      mark repeat = 20,
      mark=none,
      height=4.8cm,
      width=5.0cm,
      ymax=330,
      xmax=\limit,
      xtick={0,2,4,6,8},
      extra x ticks={\limit},
      extra x tick labels={$\lambda_{\infty}$},
      extra tick style={grid=none, grid style={dotted, cyan}},
      legend to name=legend_motivation,
      legend style={
        font=\footnotesize,
        outer sep = 0pt,
      },
      legend columns=1,
    ]

    \addplot+[
      densely dotted,
      forget plot,
    ] table [
        x index=0,
        y index=3,
        col sep=comma,
      ] {\datapath/datar8.csv};
    \addplot+[
      densely dotted,
      forget plot,
    ] table [
        x index=0,
        y index=3,
        col sep=comma,
      ] {\datapath/datar9.csv};
    \addplot+[
      densely dotted,
      forget plot,
    ] table [
        x index=0,
        y index=3,
        col sep=comma,
      ] {\datapath/datar10.csv};
    \addplot+[
      densely dotted,
      forget plot,
    ] table [
        x index=0,
        y index=3,
        col sep=comma,
      ] {\datapath/datar11.csv};
    \addplot+[
      densely dotted,
      forget plot,
    ] table [
        x index=0,
        y index=3,
        col sep=comma,
      ] {\datapath/datar12.csv};
    \addplot+[
      densely dotted,
      forget plot,
    ] table [
        x index=0,
        y index=3,
        col sep=comma,
      ] {\datapath/datar13.csv};
    \addplot+[
      densely dotted,
      forget plot,
    ] table [
        x index=0,
        y index=3,
        col sep=comma,
      ] {\datapath/datar14.csv};
    \addplot+[
      densely dotted,
      forget plot,
    ] table [
        x index=0,
        y index=3,
        col sep=comma,
      ] {\datapath/datar15.csv};
    \addplot+[
      densely dotted,
      forget plot,
    ] table [
        x index=0,
        y index=3,
        col sep=comma,
      ] {\datapath/datar16.csv};

    \addplot+[
    ] table [
        x index=0,
        y index=3,
        col sep=comma,
      ] {\datapath/datar4.csv};
    \addplot+[
    ] table [
        x index=0,
        y index=2,
        col sep=comma,
      ] {\datapath/datar4.csv};
    \addplot+[
    ] table [
        x index=0,
        y index=1,
        col sep=comma,
      ] {\datapath/datar4.csv};
    \addplot+[ 
      domain=1:1.001,
      solid,
      color=black,
      mark=none
    ]{x};
    \addplot+[ 
      domain=1:1.001,
      densely dashed,
      color=black,
      mark=none
    ]{x};
    \addplot+[ 
      domain=1:1.001,
      dashed,
      color=black,
      mark=none
    ]{x};
    \addplot+[ 
      domain=1:1.001,
      loosely dashed,
      color=black,
      mark=none
    ]{x};
    \addplot+[ 
      domain=1:1.001,
      densely dotted,
      color=black,
      mark=none
    ]{x};

    \pgfplotsset{cycle list shift=-8}
    \addplot+[
      densely dashed,
    ] table [
        x index=0,
        y index=3,
        col sep=comma,
      ] {\datapath/datar5.csv};
    \addplot+[
      densely dashed,
    ] table [
        x index=0,
        y index=2,
        col sep=comma,
      ] {\datapath/datar5.csv};
    \addplot+[
      densely dashed,
    ] table [
        x index=0,
        y index=1,
        col sep=comma,
      ] {\datapath/datar5.csv};

    \pgfplotsset{cycle list shift=-11}
    \addplot+[
      dashed,
    ] table [
        x index=0,
        y index=3,
        col sep=comma,
      ] {\datapath/datar6.csv};
    \addplot+[
      dashed,
    ] table [
        x index=0,
        y index=2,
        col sep=comma,
      ] {\datapath/datar6.csv};
    \addplot+[
      dashed,
    ] table [
        x index=0,
        y index=1,
        col sep=comma,
      ] {\datapath/datar6.csv};

    \pgfplotsset{cycle list shift=-14}
    \addplot+[
      loosely dashed,
    ] table [
        x index=0,
        y index=3,
        col sep=comma,
      ] {\datapath/datar7.csv};
    \addplot+[
      loosely dashed,
    ] table [
        x index=0,
        y index=2,
        col sep=comma,
      ] {\datapath/datar7.csv};
    \addplot+[
      loosely dashed,
    ] table [
        x index=0,
        y index=1,
        col sep=comma,
      ] {\datapath/datar7.csv};

    \legend{$\tfrac{1}{12}n_{\textnormal{tot}}$,$n_{\textnormal{CG}}$,$n_{\textnormal{IPM}}$,$L=4$,$L=5$,$L=6$,$L=7$,$L\ge8$}


  \end{axis}

  \coordinate (c3) at ($(myaxis.east)$);
  \node[right] at ($(myaxis.east)-(-0.0,0.05)$){
    \pgfplotslegendfromname{legend_motivation}
  };



\end{tikzpicture}


%% file: plots/coarse-space-sketch/factorization-ann.tex




\begin{tikzpicture}[scale=0.7, every node/.style={transform shape}]

  \pgfmathsetmacro{\iscale}{10/7}  
  \pgfmathsetmacro{\iwidth}{6.4*\iscale}
  \pgfmathsetmacro{\dist}{3.0}
  \pgfmathsetmacro{\rowDist}{2.75}
  \pgfmathsetmacro{\equSep}{1*\iscale}
  \pgfmathsetmacro{\titleSep}{1.0}
  \pgfmathsetmacro{\sep}{0.2}
  \node[inner sep=0pt] (A) at (0,0) {\includegraphics[width=\iwidth cm]{\datapath/factorization.jpg}};

  \begin{scope}[
    yshift=0,
    rotate=-30,
]
  \node at (-3.1,-3) [inner sep=0pt] {\scriptsize$\boldsymbol{\phi}$};
  \node at (-2,-2.925) [inner sep=0pt] {\scriptsize$=$};
  \node at (-1,-2.85) [inner sep=0pt] {\scriptsize$\phantom{\boldsymbol{D}_4}\boldsymbol{\psi}\phantom{\boldsymbol{R}_4}$};
  \node at (0.05,-2.775) [inner sep=0pt] {\scriptsize$\cdot$};
  \node at (1.1,-2.7) [inner sep=0pt] {\scriptsize$\phantom{\boldsymbol{D}_4}\boldsymbol{u}\phantom{\boldsymbol{R}_4}$};
  \end{scope}

\end{tikzpicture}%


%% file: plots/coarse-space-sketch/components-ann.tex




\begin{tikzpicture}[scale=0.7, every node/.style={transform shape}]

  \pgfmathsetmacro{\iscale}{10/7}  
  \pgfmathsetmacro{\iwidth}{6.4*\iscale}
  \pgfmathsetmacro{\dist}{3.0}
  \pgfmathsetmacro{\rowDist}{2.75}
  \pgfmathsetmacro{\equSep}{1*\iscale}
  \pgfmathsetmacro{\titleSep}{1.0}
  \pgfmathsetmacro{\sep}{0.2}
  \node[inner sep=0pt] (A) at (0,0) {\includegraphics[width=\iwidth cm]{\datapath/components.jpg}};

  \begin{scope}[
    yshift=0,
    rotate=-30,
]
  \node at (-3,-3) [inner sep=0pt] {\scriptsize$\boldsymbol{D}_1 \boldsymbol{R}_1 \boldsymbol{\psi}$};
  \node at (-1.63,-2.9) [inner sep=0pt] {\scriptsize$\boldsymbol{D}_2 \boldsymbol{R}_2 \boldsymbol{\psi}$};
  \node at (-0.266,-2.8) [inner sep=0pt] {\scriptsize$\boldsymbol{D}_3 \boldsymbol{R}_3 \boldsymbol{\psi}$};
  \node at (1.1,-2.7) [inner sep=0pt] {\scriptsize$\boldsymbol{D}_4 \boldsymbol{R}_4 \boldsymbol{\psi}$};
  \end{scope}%

\end{tikzpicture}%


%% file: plots/dd-geometry-sketch-tikz/dd-geometry-sketch-tikz-1.tex


\begin{tikzpicture}[scale=0.66, every node/.style={transform shape}]

  \pgfdeclarepatternformonly{soft-crosshatch-1}{\pgfqpoint{-1pt}{-1pt}}{\pgfqpoint{6pt}{6pt}}{\pgfqpoint{5pt}{5pt}}%
  {
    \pgfsetlinewidth{0.4pt}
    \pgfpathmoveto{\pgfqpoint{5.1pt}{0pt}}
    \pgfpathlineto{\pgfqpoint{0pt}{5.1pt}}
    \pgfpathmoveto{\pgfqpoint{0pt}{0pt}}
    \pgfpathlineto{\pgfqpoint{5.1pt}{5.1pt}}
    \pgfusepath{stroke}
  }

  \pgfmathsetmacro{\iscale}{10/7}  
  \pgfmathsetmacro{\iwidth}{6.4*\iscale}
  \pgfmathsetmacro{\dist}{3.0}
  \pgfmathsetmacro{\rowDist}{2.75}
  \pgfmathsetmacro{\equSep}{1*\iscale}
  \pgfmathsetmacro{\titleSep}{1.0}
  \pgfmathsetmacro{\sep}{0.2}
  \node[inner sep=0pt] (A) at (0,0) {\includegraphics[width=\iwidth cm]{\datapath/ol1b.png}};

  \node at (-1.3*\iscale,0.9*\iscale) [
    fill=white, rectangle, draw, inner sep=0pt, minimum width=28pt, minimum height=18pt
  ] {$\Omega_1'$};
  \node at (-0.4*\iscale,-0.35*\iscale) [
    preaction={clip,postaction={fill=white, draw=gray!40, line width=\iscale*1 mm}},
    rectangle, draw, inner sep=0pt, minimum width=28pt, minimum height=18pt
  ] {$\Omega_1$};
  \node at (-0.05*\iscale,-1.3*\iscale) [
    left, fill=white, rectangle, draw, inner sep=0pt, minimum width=28pt, minimum height=18pt, pattern=soft-crosshatch-1, pattern color=gray!40
  ] [preaction={fill=white}] {$\tilde{\Omega}_1$};

  \pgfmathsetmacro{\unit}{6.375*\iscale/4}
  \foreach \x in {0,1,2,3,4}
    {
      \coordinate (A\x) at ($(\x*\unit-2*\unit,-\unit)$) {};
      \draw ($(A\x)+(0,3pt)$) -- ($(A\x)-(0,3pt)$);
      \node at ($(A\x)+(0,-2ex)$) {\x};
    }
  \node at ($(-1.5*\unit,-\unit)+(0,-2ex)$) {\(x\)};
  \pgfmathsetmacro{\unit}{6.375*\iscale/4}
  \foreach \y in {0,1,2}
    {
      \coordinate (A\y) at ($(-2*\unit,\y*\unit-\unit)$) {};
      \draw ($(A\y)+(3pt,0)$) -- ($(A\y)-(3pt,0)$);
      \node at ($(A\y)+(-2ex,0)$) {\y};
    }
  \node at ($(-2*\unit,-0.5*\unit)+(-2ex,0)$) {\(y\)};

\end{tikzpicture}%


%% file: plots/dd-geometry-sketch-tikz/dd-geometry-sketch-tikz-2.tex


\begin{tikzpicture}[scale=0.66, every node/.style={transform shape}]

  \pgfdeclarepatternformonly{soft-crosshatch-2}{\pgfqpoint{-1pt}{-1pt}}{\pgfqpoint{6pt}{6pt}}{\pgfqpoint{5pt}{5pt}}%
  {
    \pgfsetlinewidth{0.4pt}
    \pgfpathmoveto{\pgfqpoint{5.1pt}{0pt}}
    \pgfpathlineto{\pgfqpoint{0pt}{5.1pt}}
    \pgfpathmoveto{\pgfqpoint{0pt}{0pt}}
    \pgfpathlineto{\pgfqpoint{5.1pt}{5.1pt}}
    \pgfusepath{stroke}
  }

  \pgfmathsetmacro{\iscale}{10/7}  
  \pgfmathsetmacro{\iwidth}{6.4*\iscale}
  \pgfmathsetmacro{\dist}{3.0}
  \pgfmathsetmacro{\rowDist}{2.75}
  \pgfmathsetmacro{\equSep}{1*\iscale}
  \pgfmathsetmacro{\titleSep}{1.0}
  \pgfmathsetmacro{\sep}{0.2}
  \node[inner sep=0pt] (A) at (0,0) {\includegraphics[width=\iwidth cm]{\datapath/ol2b.png}};

  \node at (-1.3*\iscale,0.9*\iscale) [
    fill=white, rectangle, draw, inner sep=0pt, minimum width=28pt, minimum height=18pt
  ] {$\Omega_1'$};
  \node at (-0.1*\iscale,-0.65*\iscale) [
    preaction={clip,postaction={fill=white, draw=gray!40, line width=\iscale*1 mm}},
    rectangle, draw, inner sep=0pt, minimum width=28pt, minimum height=18pt
  ] {$\Omega_1$};
  \node at (2.25-0.05*\iscale,-1.3*\iscale) [
    left, fill=white, rectangle, draw, inner sep=0pt, minimum width=28pt, minimum height=18pt, pattern=soft-crosshatch-2, pattern color=gray!40
  ] [preaction={fill=white}] {$\tilde{\Omega}_1$};

  \pgfmathsetmacro{\unit}{6.375*\iscale/4}
  \foreach \x in {0,1,2,3,4}
    {
      \coordinate (A\x) at ($(\x*\unit-2*\unit,-\unit)$) {};
      \draw ($(A\x)+(0,3pt)$) -- ($(A\x)-(0,3pt)$);
      \node at ($(A\x)+(0,-2ex)$) {\x};
    }
  \node at ($(-1.5*\unit,-\unit)+(0,-2ex)$) {\(x\)};
  \pgfmathsetmacro{\unit}{6.375*\iscale/4}
  \foreach \y in {0,1,2}
    {
      \coordinate (A\y) at ($(-2*\unit,\y*\unit-\unit)$) {};
      \draw ($(A\y)+(3pt,0)$) -- ($(A\y)-(3pt,0)$);
      \node at ($(A\y)+(-2ex,0)$) {\y};
    }
  \node at ($(-2*\unit,-0.5*\unit)+(-2ex,0)$) {\(y\)};

\end{tikzpicture}%


%% file: plots/iteration-numbers/iteration-numbers.tex



\begin{tikzpicture}

  \def\LxList{2,4,8,16,32,64,128,256}

  \definecolor{color1}{rgb}{0,    ,0.4470,0.7410}
  \definecolor{color2}{rgb}{0.8500,0.3250,0.0980}
  \definecolor{color3}{rgb}{0.9290,0.6940,0.1250}
  \definecolor{color4}{rgb}{0.4940,0.1840,0.5560}
  \definecolor{color5}{rgb}{0.4660,0.6740,0.1880}
  \definecolor{color6}{rgb}{0.3010,0.7450,0.9330}
  \definecolor{color7}{rgb}{0.6350,0.0780,0.1840}
  \pgfplotscreateplotcyclelist{matlab}{
    color1,every mark/.append style={solid},mark=*\\
    color2,every mark/.append style={solid},mark=square*\\
    color3,every mark/.append style={solid},mark=triangle*\\
    color4,every mark/.append style={solid},mark=halfsquare*\\
    color5,every mark/.append style={solid},mark=pentagon*\\
    color6,every mark/.append style={solid},mark=halfcircle*\\
    color7,every mark/.append style={solid,rotate=180},mark=halfdiamond*\\
    color1,every mark/.append style={solid},mark=diamond*\\
    color2,every mark/.append style={solid},mark=halfsquare right*\\
    color3,every mark/.append style={solid},mark=halfsquare left*\\
  }

\begin{groupplot}[
  group style={
    group size= 2 by 1,
    horizontal sep=0.15cm,
    vertical sep=0.45cm,
  },
  ymin=1E-8,
  ymax=10E0,
  xmode=linear,
  ymode=log,
  legend style={
    font=\normalsize,
    nodes={
      scale=0.75,
      transform shape
    }
  },
  legend pos=south west,
  legend columns=9,
  height=5.0cm,
  width=7.15cm,
  cycle list name=matlab,
  ticklabel style = {font=\footnotesize},
  grid=major,
  xlabel={iteration $k$},
  xmin=-1,
  xmax=51,
]

  \nextgroupplot[
    title={CG + AS1},
    legend to name=legend_iteration_numbers_123,
    mark size=1.5pt,
  ]
    \foreach \Lx in \LxList {
      \addplot+[
      ] table [
        x index=0,
        y index=1,
        col sep=comma,
      ] {\datapath/cg_asm1_30-1-\Lx.csv};
      \addlegendentryexpanded{$L=\Lx$}
    }

  \nextgroupplot[
    title={CG + AS2},
    yticklabels={},
    ylabel={rel.~res.~norm ${\|\boldsymbol{r}_k\|}_2/{\|\boldsymbol{r}_0\|}_2$},
    yticklabel pos=right,
    mark size=1.5pt,
  ]
  \foreach \Lx in \LxList {
    \addplot+[
    ] table [
      x index=0,
      y index=1,
      col sep=comma,
    ] {\datapath/cg_asm2_30-1-\Lx.csv};
  }

\end{groupplot}

\coordinate (c3) at ($(group c1r1)!.5!(group c2r1)$);
\node[below] at (c3 |- current bounding box.south)
  {\pgfplotslegendfromname{legend_iteration_numbers_123}};

\end{tikzpicture}


%% file: plots/fused-iteration-numbers/it-nums.tex



\begin{tikzpicture}

  \def\LxList{4,8,16}
  \def\methodList{fixed,adaptive,fused}

  \definecolor{color1}{rgb}{0,    ,0.4470,0.7410}
  \definecolor{color2}{rgb}{0.8500,0.3250,0.0980}
  \definecolor{color3}{rgb}{0.9290,0.6940,0.1250}
  \definecolor{color4}{rgb}{0.4940,0.1840,0.5560}
  \definecolor{color5}{rgb}{0.4660,0.6740,0.1880}
  \definecolor{color6}{rgb}{0.3010,0.7450,0.9330}
  \definecolor{color7}{rgb}{0.6350,0.0780,0.1840}
  \pgfplotscreateplotcyclelist{matlab}{
    color1,every mark/.append style={solid},mark=*\\
    color2,every mark/.append style={solid},mark=square*\\
    color3,every mark/.append style={solid},mark=triangle*\\
    color4,every mark/.append style={solid},mark=halfsquare*\\
    color5,every mark/.append style={solid},mark=pentagon*\\
    color6,every mark/.append style={solid},mark=halfcircle*\\
    color7,every mark/.append style={solid,rotate=180},mark=halfdiamond*\\
    color1,every mark/.append style={solid},mark=diamond*\\
    color2,every mark/.append style={solid},mark=halfsquare right*\\
    color3,every mark/.append style={solid},mark=halfsquare left*\\
  }

\begin{groupplot}[
  group style={
    group size= 4 by 1,
    horizontal sep=0.15cm,
    vertical sep=0.45cm,
  },
  ymin=1E-8,
  ymax=10E0,
  xmode=linear,
  ymode=log,
  legend style={
    font=\normalsize,
    nodes={
      scale=0.75,
      transform shape
    }
  },
  legend pos=south west,
  legend columns=9,
  height=3.5cm,
  width=4.25cm,
  cycle list name=matlab,
  ticklabel style = {font=\footnotesize},
  grid=major,
  xlabel={inner it.~$k$},
  xmin=-1,
  xmax=399,
]

  \nextgroupplot[
    title={\(L=4\)},
    legend to name=legend_it_nums_231110,
    mark size=1.0pt,
  ]
  \foreach \method in \methodList {
    \addplot+[
    ] table [
      x index=0,
      y index=1,
      col sep=comma,
    ] {\datapath/its_\method_4.csv};
  }
  \legend{GMRES+RAS2 (fixed $\texttt{rTOL}_{\textnormal{i}}$), GMRES+RAS2 (adaptive $\texttt{rTOL}_{\textnormal{i}}$), \(\boldsymbol{M}_{\textnormal{RAS,2}}^{-1}\)}

  \nextgroupplot[
    title={\(L=8\)},
    yticklabels={},
    mark size=1.0pt,
  ]
  \foreach \method in \methodList {
    \addplot+[
    ] table [
      x index=0,
      y index=1,
      col sep=comma,
    ] {\datapath/its_\method_8.csv};
  }

  \nextgroupplot[
    title={\(L=16\)},
    yticklabels={},
    mark size=1.0pt,
  ]
  \foreach \method in \methodList {
    \addplot+[
    ] table [
      x index=0,
      y index=1,
      col sep=comma,
    ] {\datapath/its_\method_16.csv};
  }

  \nextgroupplot[
    title={\(L=32\)},
    yticklabels={},
    ylabel={spectral.~res.~${\|\boldsymbol{r}_{\textnormal{o},k}\|}_2$},
    yticklabel pos=right,
    mark size=1.0pt,
  ]
  \foreach \method in \methodList {
    \addplot+[
    ] table [
      x index=0,
      y index=1,
      col sep=comma,
    ] {\datapath/its_\method_32.csv};
  }

\end{groupplot}

\coordinate (c3) at ($(group c1r1)!.5!(group c4r1)$);
\node[below] at (c3 |- current bounding box.south)
  {\pgfplotslegendfromname{legend_it_nums_231110}};

\end{tikzpicture}
